\documentclass[12pt]{article}

\usepackage[T1]{fontenc}          
\usepackage[utf8]{inputenc}

\usepackage{amsmath,amsfonts,amssymb,amsthm,mathtools}
\usepackage{mathrsfs}             
\usepackage{bm}

\usepackage[a4paper]{geometry}
\usepackage{graphicx}
\usepackage{float}
\usepackage{placeins}
\usepackage{caption}
\usepackage{marginnote}
\usepackage{comment}

\usepackage{array}
\usepackage{multirow}
\usepackage{booktabs}
\usepackage{longtable}
\usepackage{tabularx}
\usepackage{needspace}

\usepackage[shortlabels]{enumitem}

\usepackage{algorithm}
\usepackage{algpseudocode}

\usepackage{tikz}
\usetikzlibrary{positioning}
\usetikzlibrary{calc}

\usepackage[unicode]{hyperref}
\hypersetup{
	colorlinks=true,
	linkcolor=blue,
	citecolor=blue,
	urlcolor=blue
}

\numberwithin{figure}{section}
\numberwithin{equation}{section}

\makeatletter

\newtheoremstyle{myplain}
{6pt}{6pt}%
{\itshape}
{}%
{\bfseries}
{.}%
{0.5em}%
{}%

\newtheoremstyle{mydefinition}
{6pt}{6pt}%
{\normalfont}
{}%
{\bfseries}
{.}%
{0.5em}%
{}%

\newtheoremstyle{myremark}
{6pt}{6pt}%
{\normalfont}
{}%
{\bfseries}
{.}%
{0.5em}%
{}%

\renewenvironment{proof}[1][Proof]{%
	\par\pushQED{\qed}\normalfont
	\topsep6pt\relax
	\trivlist
	\item[\hskip\labelsep\bfseries #1.]%
}{%
	\popQED\endtrivlist\@endpefalse
}

\makeatother

\theoremstyle{myplain}
\newtheorem{theorem}{Theorem}[section]

\newtheorem{lemma}[theorem]{Lemma}

\theoremstyle{mydefinition}

\newtheorem{algo}[theorem]{Algorithm}

\theoremstyle{myremark}

\theoremstyle{myplain}
\newtheorem{theo}[theorem]{Theorem}
\newtheorem{statm}[theorem]{Statement}
\newtheorem{cor}[theorem]{Corollary}

\theoremstyle{mydefinition}
\newtheorem{defi}[theorem]{Definition}

\theoremstyle{myremark}
\newtheorem{rmrk}[theorem]{Remark}

\newcommand{\capa}{\operatorname{cap}}    
\newcommand{\ch}{\operatorname{ch}}       
\newcommand{\vcount}{\operatorname{vcount}}

\begin{document}
\title{Another proof of Cruse’s theorem and a new necessary condition for completion of partial Latin squares (Part 3.)\footnote{This is the third part of our consecutive papers on Latin and partial Latin squares and on bipartite graphs. Each work holds the terminology and notation of the previous ones~\cite{[3],[4]}.}
}
\author{Béla Jónás}
\date{Revised August 2026 \hspace{2pt} Budapest, Hungary}

\maketitle

\begin{abstract}
A \emph{partial Latin square} of order $n$ can be represented by a $3$-dimensional chess-board of size $n\times n\times n$ with at most $n^2$ non-attacking rooks. 
Based on this representation, we give proofs of the theorems of M.~Hall, Ryser and Cruse on the completion of partial Latin squares that share a common device, the \emph{cover sheet}: in each case the cover sheet is extended to a $(0,1)$-matrix with constant line sums and decomposed into permutation matrices by K\H{o}nig's theorem.
With the help of this proof, we extend the scope of Cruse's theorem to \emph{compact} bricks, which appear to be independent of their environment.

Without losing any completion you can replace a dot by a rook if the dot must become rook, or you can eliminate the dots that are known not to become rooks. Therefore, we introduce \emph{primary} and \emph{secondary extension} procedures that are repeated as many times as possible. If the procedures do not decide whether a PLSC can be completed or not, a new necessary condition for completion can be formulated for the dot structure of the resulting PLSC, the \emph{BUG condition}.
\end{abstract}

\textbf{MSC-Class:} 05B15

\textbf{Keywords:} Latin square, partial Latin square 

\section*{Abbreviations}
\begin{tabular}{@{}ll}
LS&Latin square\\
LSC&Latin super cube\\
$d$-LSC&Latin super cube of dimension $d$\\
PLS&partial Latin square\\
PLSC&partial Latin super cube\\
RBC&remote brick couple\\
RAC&remote axis couple\\
BUG&Bivalue Universal Grave
\end{tabular}

\section{Introduction}
Generally, we define a set of candidates $P_{ij}$ for each empty cell $(i,j)$ of a PLS as follows: $k \in P_{ij}$ exactly if the cell $(i,j,k)$ of the proper PLSC has a dot. A PLSC can be projected to a PLS in the usual way, the only difference is that each empty cell of the PLS is connected to a set of candidates.
The square of Goldwasser, Hilton and Patterson~\cite{[7]}, also discussed in~\cite{[4]}, will be referred to as \emph{GW6} hereafter. When illustrating a smaller PLS, such as \emph{GW6}, the candidates are written in the empty cells, but with smaller numbers, as is usual when solving Sodoku puzzles. The \emph{GW6} prepared in this way is depicted in the Figure~\ref{fig1_1}.
So, the 35 in the cell (5,4) means that $P_{54} = \{35\}$.

\begin{figure}[!htb]
\centering\includegraphics[width=0.4\textwidth]
{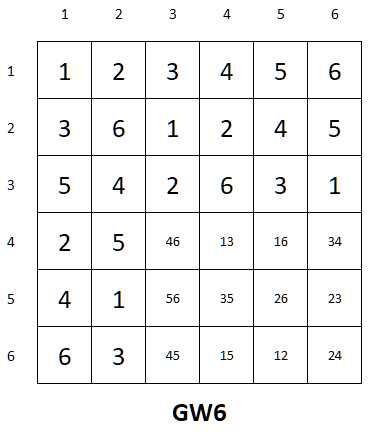}
\caption{}\label{fig1_1}
\end{figure}

\section{Basic Results of Completion}
\begin{defi} 
Let $P$ be a PLSC and $T$ a brick of size $r\times s\times t$ containing all the rooks of $P$, where $T$ is permuted to the origin, $t < n$ and $t$ is on the symbol axis. The cover sheet of $T$ is the matrix of size $n\times n$, whose cell with coordinate $(i,j)$ is one if there are no rooks in the cells of $T$ with coordinate $(i,j,c)$, where $c \in \{1,2,\ldots ,t\}$, and zero otherwise. So, the matrix is partly filled, as illustrated in the Figure~\ref{fig5_1}
\end{defi}

\begin{figure}[!htb]
\centering
\begin{tabular}{c}
\hbox to 0.95\textwidth {\includegraphics[width=0.44\textwidth]{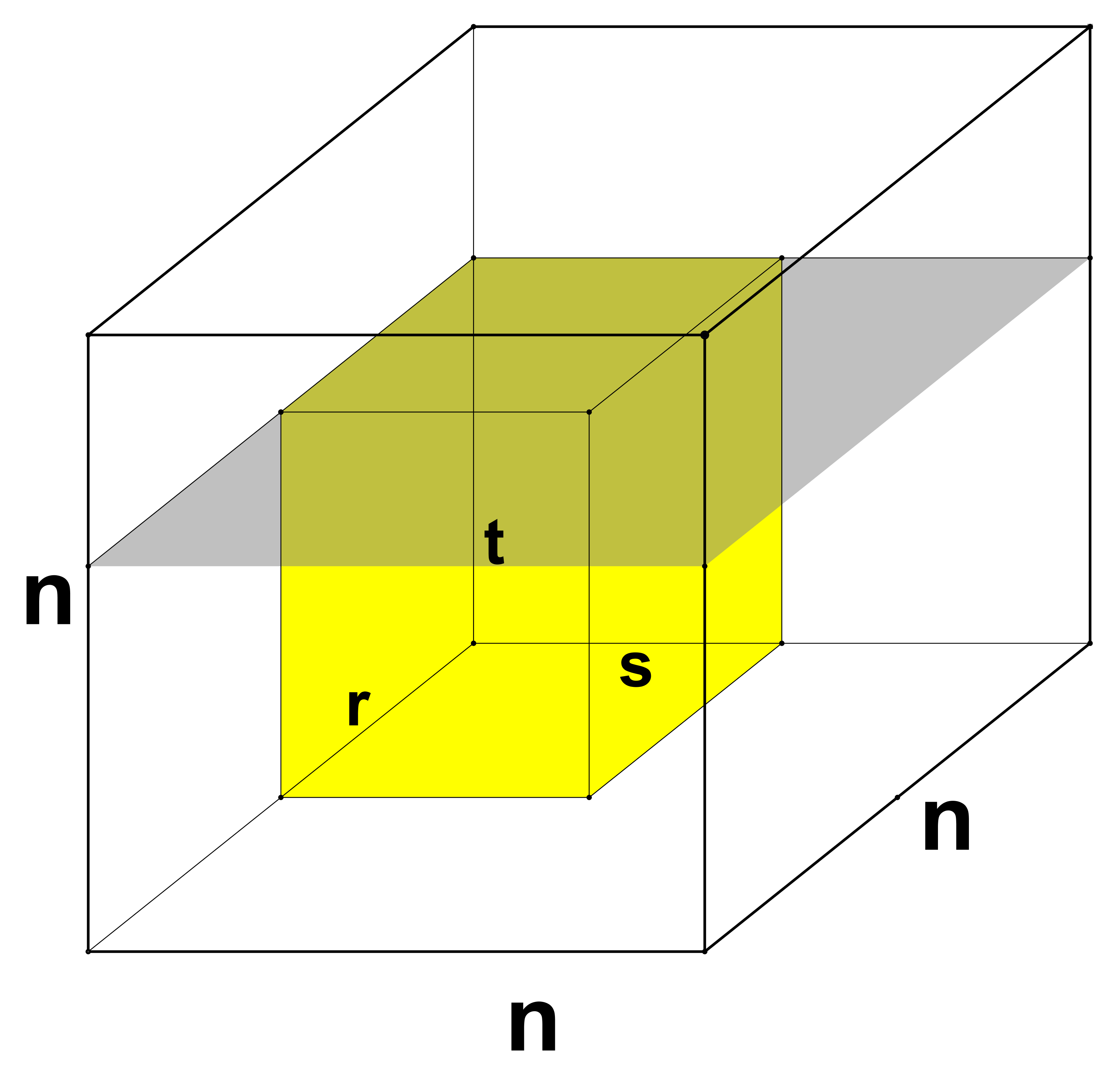}\hfill\includegraphics[width=0.44\textwidth]{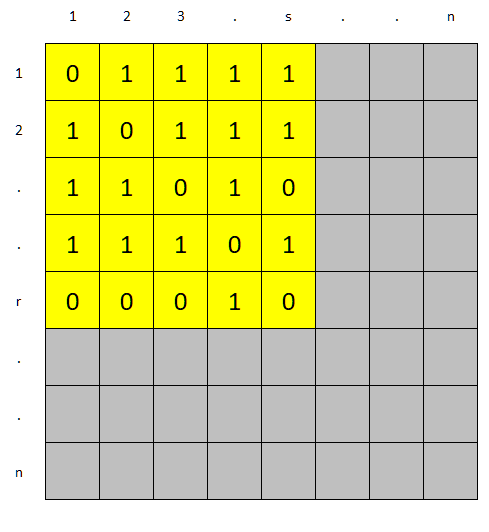}}
\end{tabular}
\caption{}\label{fig5_1}
\end{figure}

The proofs of the theorems of M.~Hall, Ryser and Cruse given below share a common device: in each case we construct a cover sheet, extend it to a $(0,1)$-matrix of size $n\times n$ with constant line sums, and decompose that matrix into a sum of permutation matrices by K\H{o}nig's theorem~\cite{[5]}. The three proofs differ in what is needed to carry out the extension: for M.~Hall's theorem no extension is needed, for Ryser's theorem we invoke his combinatorial theorem~\ref{theo1.1}, and for Cruse's theorem we prove a generalization of it, Theorem~\ref{theo1.6}.

Let $P$ be a PLSC of order $n$ and suppose that $P$ consists of some completely filled column layers, row layers or symbol layers permuted next to each other. The Figure~\ref{fig1_3} depicts these yellow  $n^2$-bricks that  contain $sn$, $rn$, and $tn$ rooks, respectively.

\begin{figure}[htb]
\centering
\hbox to \textwidth {\includegraphics[width=0.3\textwidth]{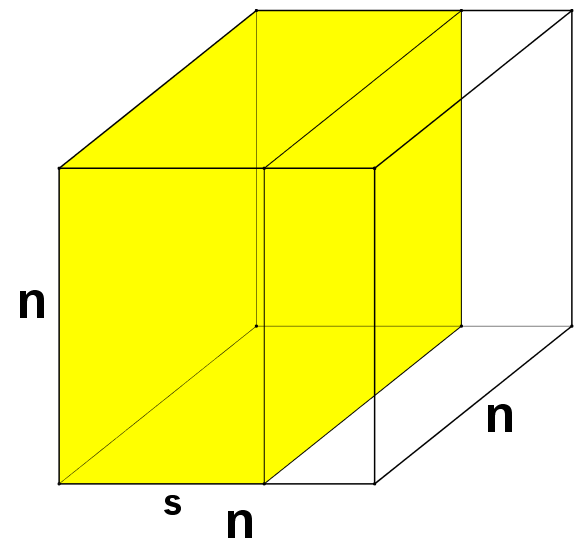}\hfill\includegraphics[width=0.3\textwidth]{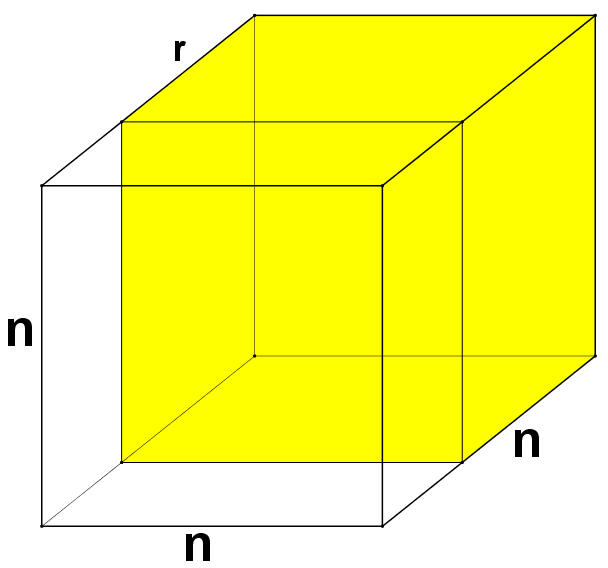}
\hfill\includegraphics[width=0.3\textwidth]{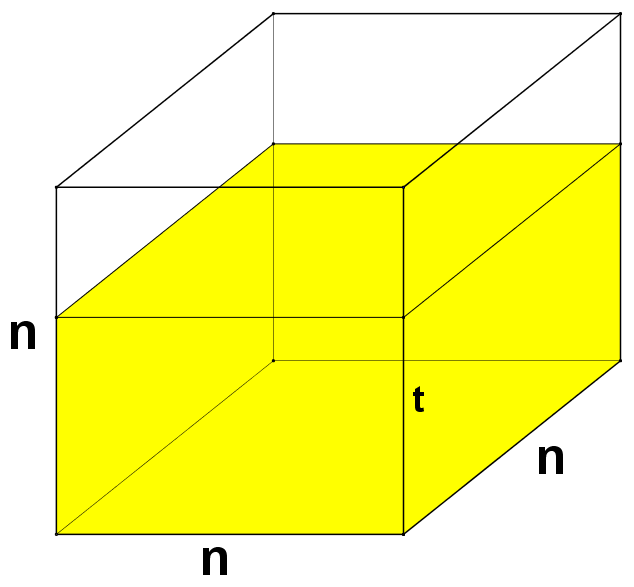}}
\caption{}\label{fig1_3}
\end{figure}

Let $P$ be a PLSC where the first $r$ row layers are completed. The projected version of this PLSC is a Latin rectangle of size $r\times n$. An example of a Latin rectangle of size $r\times n$ and the two conjugates of this PLS are shown in the Figure~\ref{fig1_2}. In our case, they are conjugates of each other, so $r = s = t = 5$ for the corresponding PLSCs in the Figure~\ref{fig1_3}.

\begin{figure}[htb]
\centering
\includegraphics[width=0.95\textwidth]{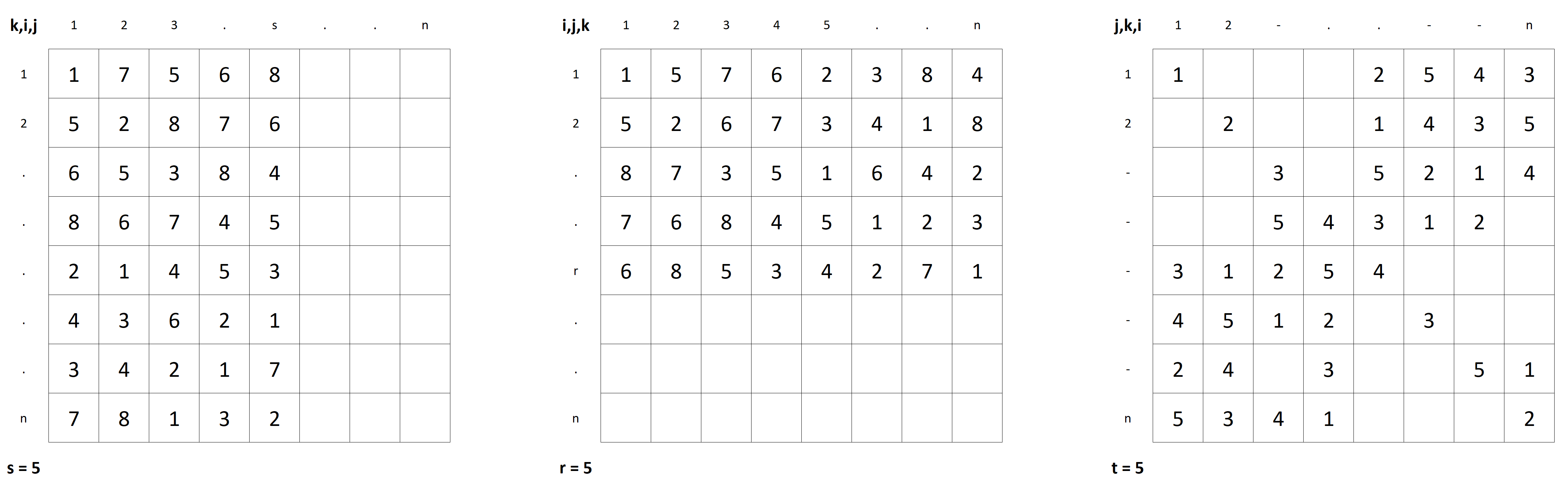}
\caption{}\label{fig1_2}
\end{figure}

In case of $n^2$-bricks the matrix is completely filled and the example in the right-hand side of Figure~\ref{fig1_4} is the cover sheet of the PLS shown on the right-hand side of Figure~\ref{fig1_2}.

\begin{figure}[!htb]
\centering
\begin{tabular}{c}
\hbox to 0.95\textwidth {\includegraphics[width=0.44\textwidth]{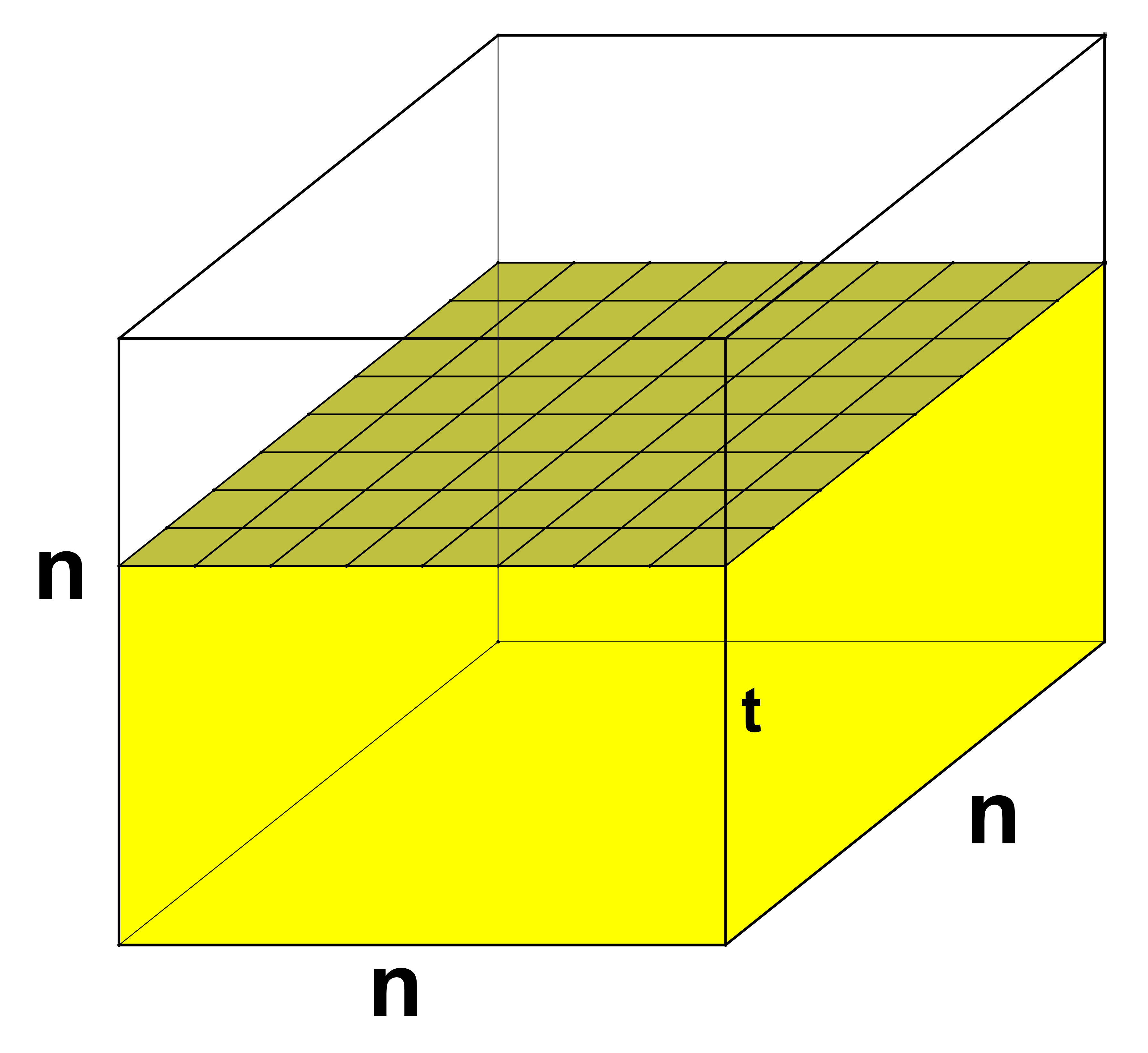}\hfill\includegraphics[width=0.44\textwidth]{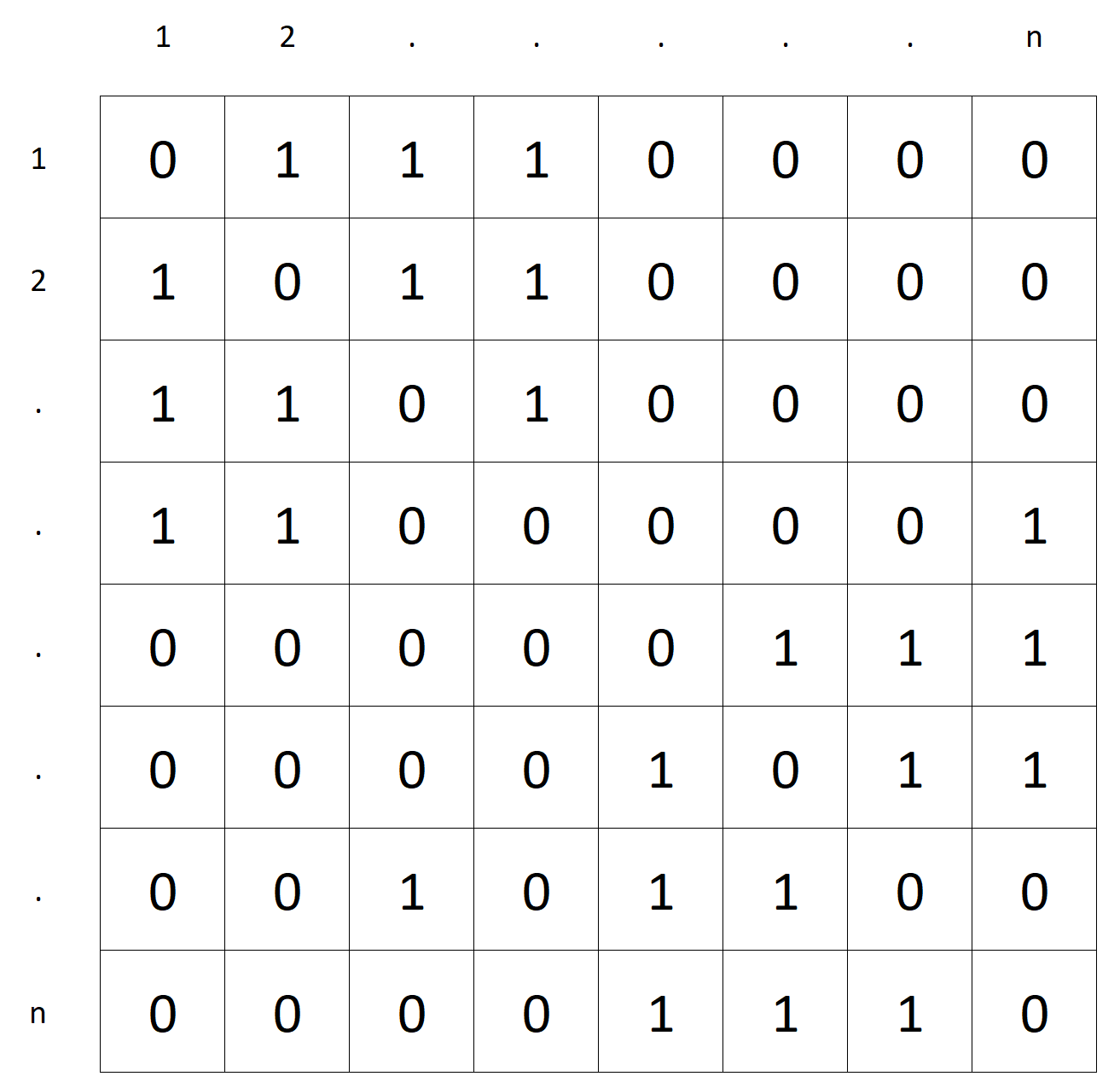}}
\end{tabular}
\caption{}\label{fig1_4}
\end{figure}

\begin{theo}[M.~Hall~\cite{[2]}]
Every $r\times n$ Latin rectangle, $1 \leq r < n$, can be completed to a Latin square of order $n$.
\end{theo}

\begin{proof}
Let $P$ be the PLSC form of the conjugated PLS denoted by $(j,k,i)$. Then each of the layers $1,2,\ldots ,t$ contains $n$  rooks. Let the cover sheet of $P$ be the matrix $M$. Each row and each column of $M$ contains exactly $n-t$  ones. Then, based on the Kőnig's theorem~\cite{[5]}, the matrix $M$ decomposes into the sum of $(n-t)$ permutation matrices, that is
\[
M = L_{t+1} + L_{t+2} + \ldots + L_n
\]
where $L_k$ is a permutation matrix for all $k \in \{t+1,t+2,\ldots ,n\}$. 
Put a rook in the cell of the PLSC $P$ with coordinate $(i,j,k)$ if the element of the matrix $L_k$ with coordinate $(i,j)$ is one, where $k \in \{t + 1, t + 2,\ldots ,n\}$. Then each layer of $P$ contains $n$  rooks, and all of the rooks are non-attacking, so we get a completion of $(j,k,i)$ conjugate of $P$.
The $(i,j,k)$-conjugate of this completion is a completion of $P$.
\end{proof}

\begin {rmrk}
If a PLSC $P$ has a completion, then each element of the main class of $P$ also has a completion.
\end{rmrk}

To prove his theorem Ryser first proved a combinatorial theorem.

\begin{theo} [Ryser's combinatorial theorem] \label{theo1.1}
Let $A$ be a matrix of $r$ rows and $n$ columns, composed entirely of zeros and ones, where $1 \leq r < n$. Let there be exactly $k$ ones in each row, and let $N(i)$ denote the number of ones in the $i$-th column of $A$.
If, for each $i=1,2,\ldots n$ holds, that
\[
k-(n-r) \leq N(i) \leq k
\]
then $n - r$ rows of zeros and ones may be adjoined to $A$ to obtain a square matrix with exactly $k$  ones in each row and in each column.
\end{theo}

and then
\begin{theo}[Ryser~\cite{[6]}]
Let $T$ be an $r$ by $s$ Latin rectangle based upon the integers $1,2,\ldots ,n$. Let $N(i)$ denote the number of times that the integer $i$ occurs in $T$. A necessary and sufficient condition in order that $T$ may be extended to a Latin square of order $n$ is that
\[
N(i) \geq r + s - n
\]
for each $i \in \mathbb{Z}_n$.
\end{theo}

Using Theorem~\ref{theo1.1} we prove with the help of cover sheet.

\begin{proof}
Let $P$ be a PLSC derived from the conjugate of the Latin rectangle $T$ denoted by $(k,i,j)$. Then each column layer of $P$ contains exactly $s$ rooks and each rook is in the layers $1,2,\ldots ,s$, as shown in the Figure~\ref{fig1_5} and $N(i) \geq r+s-n$ for all $i = 1,2,\ldots ,n$ due to the Ryser’s condition. Take the cover sheet of the yellow brick $T$. The zeros do not matter, so they are hereafter omitted. For the matrix $M$ holds that each of the first $r$ columns contains $n-s$ ones. In the row layer $i$ of $P$ there are $N(i)$ rooks and obviously $N(i) \leq s$. Thus, the matrix $M$ contains $R(i) = r - N(i)$ ones in each row, as the example shows in the Figure~\ref{fig1_6}.
Put $\kappa = n-s$. Then there are $\kappa$ ones in each column of the matrix $M$, and each row contains $R(i) = r - N(i)$ ones. (In Theorem~\ref{theo1.1} the number of ones per line is denoted by $k$; here that role is played by $\kappa$, since $k$ is reserved for the conjugate index.)

\begin{figure}[!htb]
\centering
\begin{tabular}{c}
\hbox to 0.95\textwidth {
\includegraphics[width=0.44\textwidth]{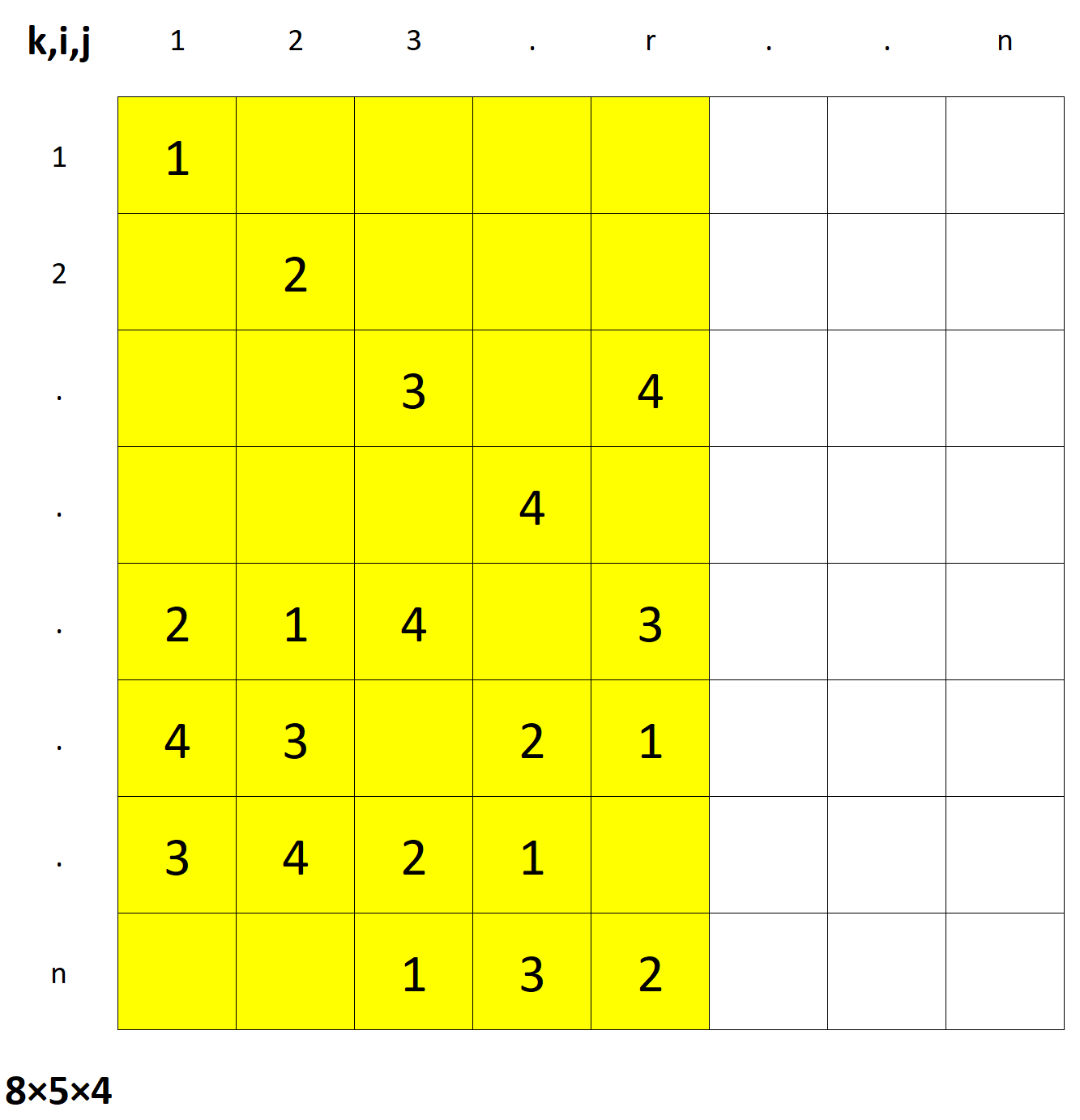}\hfill\includegraphics[width=0.44\textwidth]{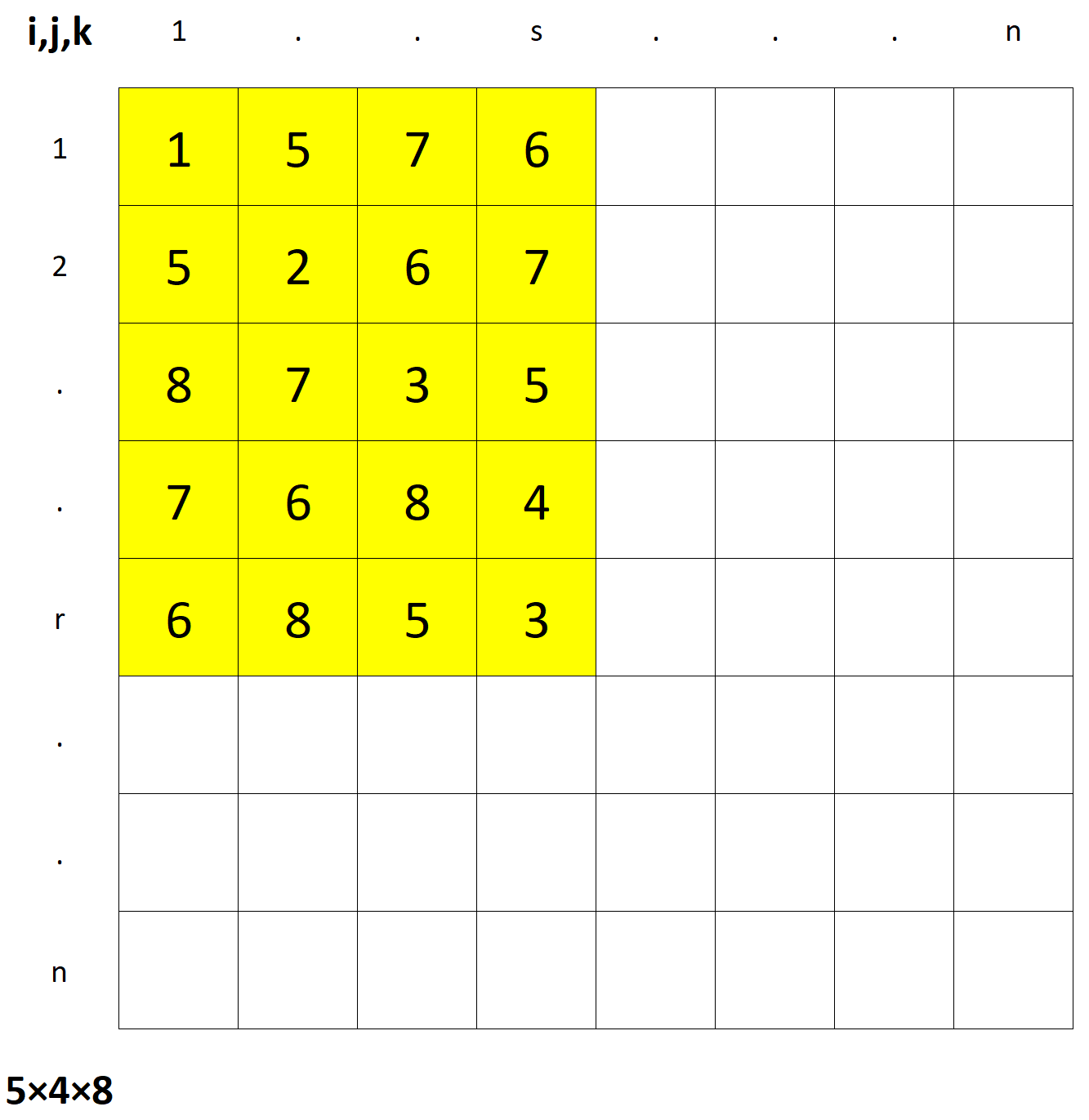}}
\end{tabular}
\caption{}\label{fig1_5}
\end{figure}

Because of the Ryser’s condition 
\[
N(i) \geq r+s-n  \Rightarrow  r-R(i) \geq r+s-n  \Rightarrow  \kappa = (n-s) \geq R(i)
\]
Since each row contains at most $\kappa$ ones and because of $N(i) \leq s$ it holds that
\[
(n-s) \leq n - N(i) \Rightarrow \kappa = (n-s) \leq n - N(i) = (n-r) + R(i)
\]
ergo $\kappa - (n-r) \leq R(i)$.
That is, based on Theorem~\ref{theo1.1}, the grey part of $M$ can be filled by ones such that each row and each column of $M$ contains exactly $\kappa$ ones, as indicated in the Figure~\ref{fig1_6}.

\begin{figure}[htb]
\centering
\hbox to \textwidth {\includegraphics[width=0.3\textwidth]{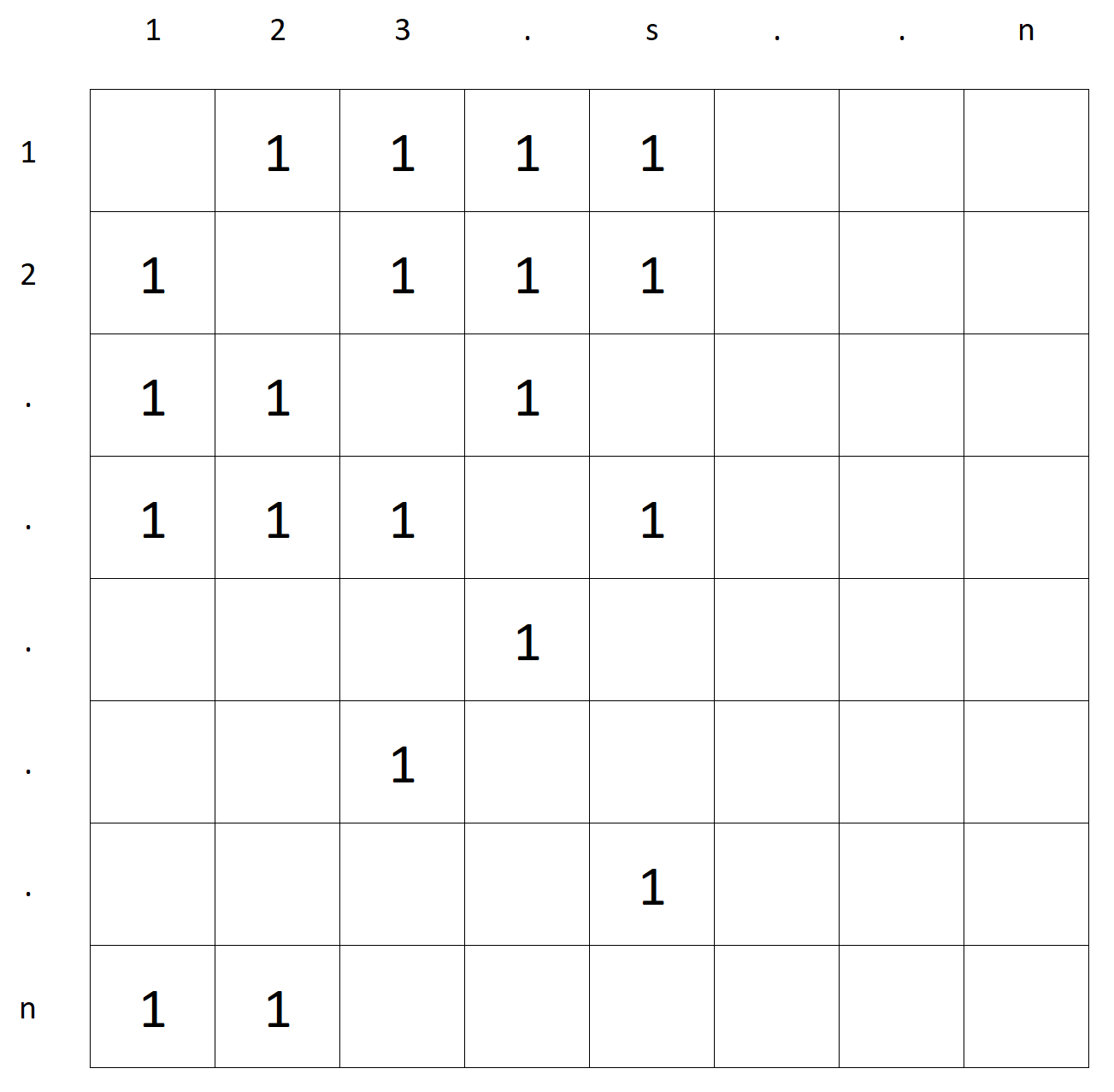}\hfill\includegraphics[width=0.3\textwidth]{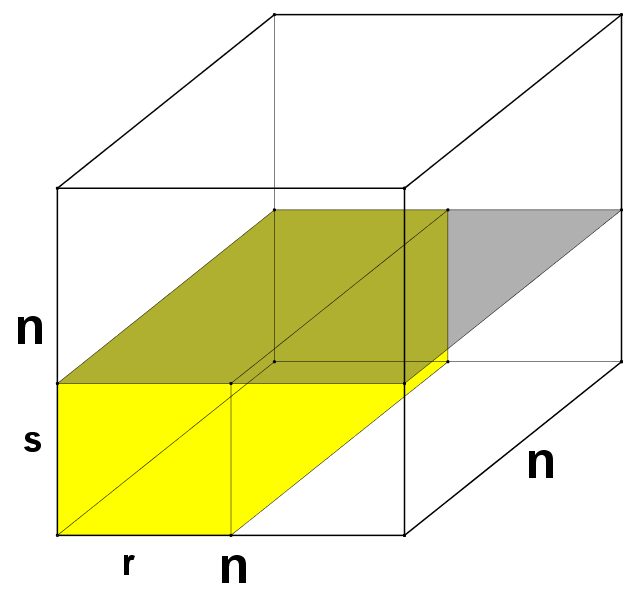}
\hfill\includegraphics[width=0.3\textwidth]{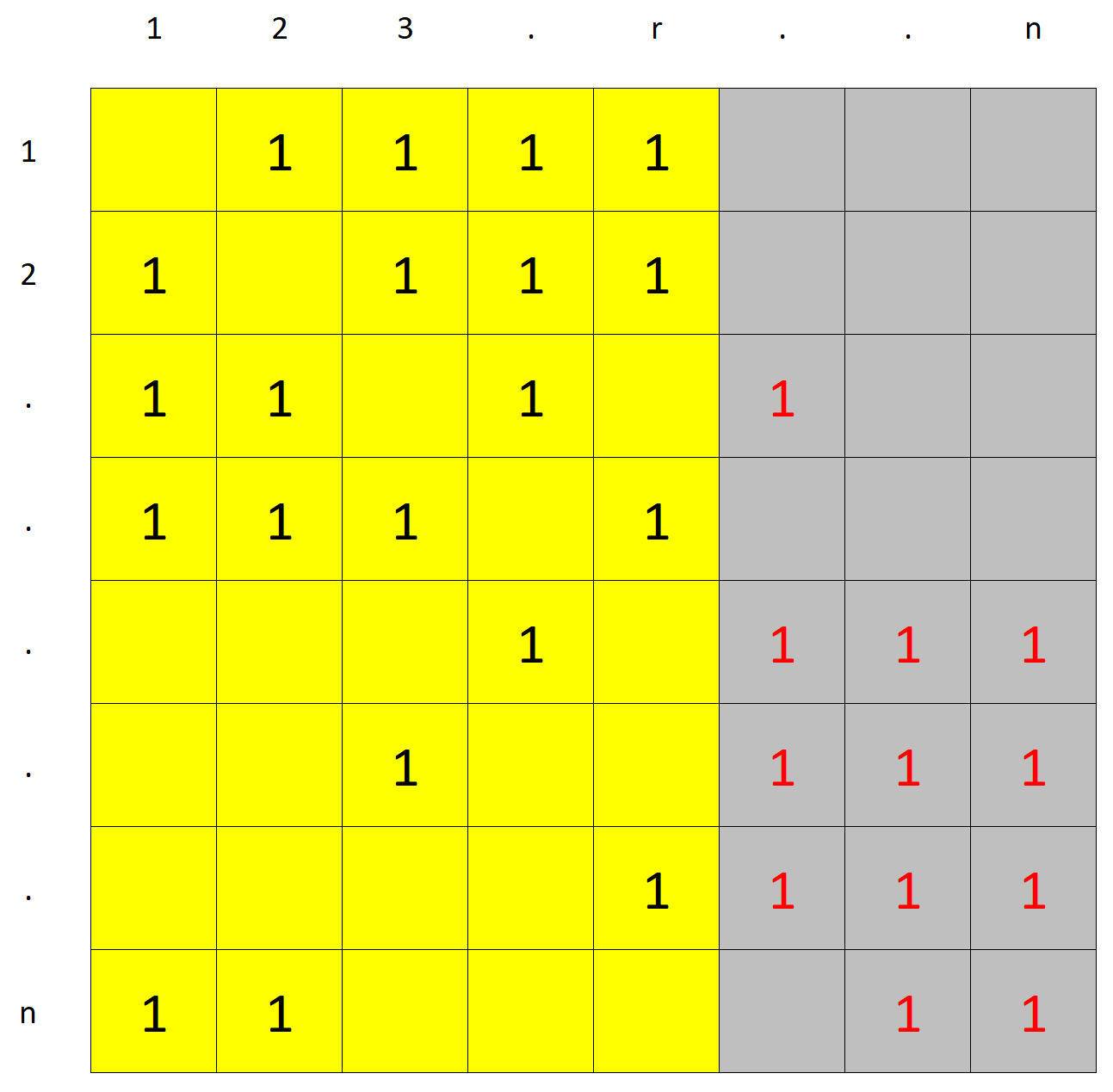}}
\caption{}\label{fig1_6}
\end{figure}

Having constant line sums $\kappa = n-s$, the matrix $M$ decomposes by K\H{o}nig's theorem~\cite{[5]} into $\kappa$ permutation matrices,
\[
M = L_{s+1} + L_{s+2} + \ldots + L_n .
\]
Let $T_1$ be the vertical auxiliary brick of $T$ (green brick) and $T_2$ be the horizontal auxiliary brick of $T_1$ (grey brick). Now you can fill the layer $\ell$ based on $L_\ell$, where $\ell = s+1,s+2,\ldots ,n$, therefore $(T_1\cup T_2)$ and $(T\cup T_1)$ form completed $n^2$-bricks, as depicted in the Figure \ref{fig1_7}. Removing the "new" rooks from the brick $T_2$, the $n^2$-brick $(T\cup T_1)$ can be completed by M.~Hall's theorem~\cite{[2]}. The corresponding conjugate of the result is a completion of the original (yellow) brick $T$.

\begin{figure}[!htb]
\centering
\includegraphics[width=0.6\textwidth] {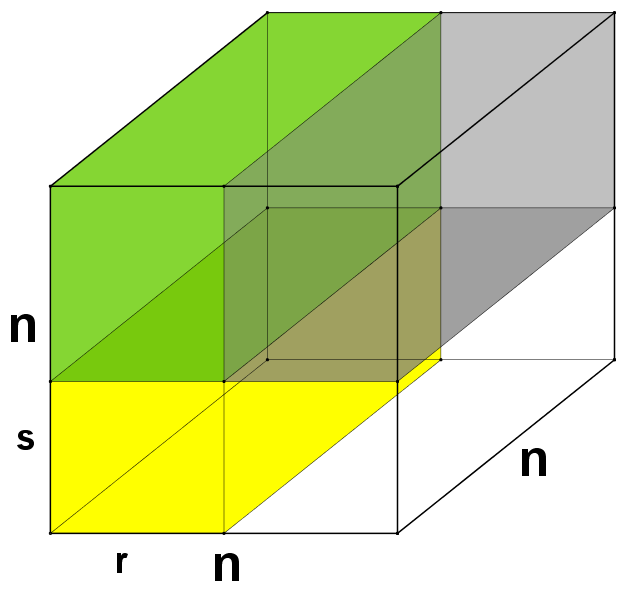}
\caption{}\label{fig1_7}
\end{figure}
\end{proof}
With the help of the method used in the previous proof, we prove the Cruse’s theorem.
 
\begin{theo} [Cruse's theorem] \label{theo1.4} 
Let $P$ be a PLSC in $H_n^3$ for which $\ch(P)$ is a real brick $T_0$ of size $r\times s\times t$ ($r,s,t < n$) with $c_0$ rooks. The brick $T_0$ can be embedded in a Latin square of order $n$  exactly if the following four conditions are met:
\begin{enumerate} [1)]
\item\label{251} 
Each row    layer of $T_0$ contains at least $s+t-n$ rooks
\item\label{252} 
Each column layer of $T_0$ contains at least $r+t-n$ rooks
\item\label{253} 
Each symbol layer of $T_0$ contains at least $r+s-n$ rooks
\item\label{254} 
$c_0 \leq \capa(n,r,s,t)$ 
\end{enumerate}
\end{theo}

Conditions 1)-4) are collectively referred to as the \emph{Cruse’s condition system}.
The conditions are obviously necessary due to the Distribution Theorem. We only prove that they are sufficient.
First, we prove a lemma.
\begin{lemma}\label{lemma1.5}
In this lemma, and only here, $T$ is indexed by columns first: $T$ has $p$ columns and $q$ rows.
Let $T$ be a matrix with $p$ columns and $q$ rows and let $0 \leq c_1, c_2,\ldots , c_p \leq q$ be integers. Suppose that $N = c_1 + c_2 +\ldots + c_p$ and for integers $m$ and $M$ it holds that $0 \leq m, M \leq p$ and $mq \leq N \leq Mq$. Then $N$ ones can be placed in the cells of $T$ so that there are exactly $c_j$ ones in the column $j$ for each $j \in \{1,2,\ldots, p\}$, and $m \leq N(i) \leq M$ for each $i \in \{1,2,\ldots, q\}$, where $N(i)$ denotes the number of ones in the row $i$.
\end{lemma}
\begin{proof} 
For each $j \in \{1,2,\ldots, p\}$ place $c_j$ ones in the column $j$ from bottom to top and permute the columns so that the number of ones contained does not increase compared to the number in the previous column, shown in the left-hand side of Figure~\ref{fig1_8}.
Notice, that if there is a one in a cell, then there is a one in each cell below it in its column and in each cell on the left in its row, so the rectangle to the lower left corner is full of ones. If there are at most $M$ ones in the bottom row, then there are at most $M$ ones in each row, one side of the inequality holds. If there are more than $M$ ones in the bottom row, we push the ones of the columns $M+1,\ldots ,p$ upwards by 1 (green cells on the left-hand side of the Figure~\ref{fig1_8}). 
There is no ones in the first row of these columns (yellow cells in the left-hand side of the Figure~\ref{fig1_8}), otherwise there would be a rectangle of size $(M+1)\times q$ full of ones. Since $N - M \leq (q-1)M$, the procedure can be continued with the subsystem $T^*$ formed by disregarding the bottom line containing exactly $M$ ones. We reconstruct the monotonic decrease of the columns in $T^*$ by permuting the whole columns (including the cells below $T^*$) and continue the procedure with $T^*$. 
Finally, $N(i) \leq M$ for all $i \in \{1,2,\ldots ,q\}$.

Starting from the resulting structure of the previous procedure, we produce the matrix for which the other side of the inequality holds. Permute the rows so that the number of ones contained increases monotonically, with the fewest ones in the first row. If $N(1) \geq m$, then we are ready. So, $N(1) < m$. Permute the columns so that the ones in the first row are to the left. In the bottom row, going from right to left, find the first cell that contains a one (green cells on the right-hand side of the Figure~\ref{fig1_8}). If there is no one in the lowest green cells of the columns $p, p-1,\ldots ,m + 1$, then the bottom row and thus each row contains at most $m$ ones, and the first row contains less, so it is a contradiction. Push the one found in the last row up into the first row, then continue the procedure from the beginning with the whole system thus obtained.
Finally, $m \leq N(i)$ for all $i \in \{1,2,\ldots , q\}$.
\end{proof}
\begin{rmrk}
If $N/q$ is an integer, then it is possible that each row contains the same number of ones, namely $m = M = N/q$. 
\end{rmrk}
\begin{figure}[!htb]
\centering
\begin{tabular}{c}
\hbox to .97\textwidth {\includegraphics[width=0.44\textwidth]{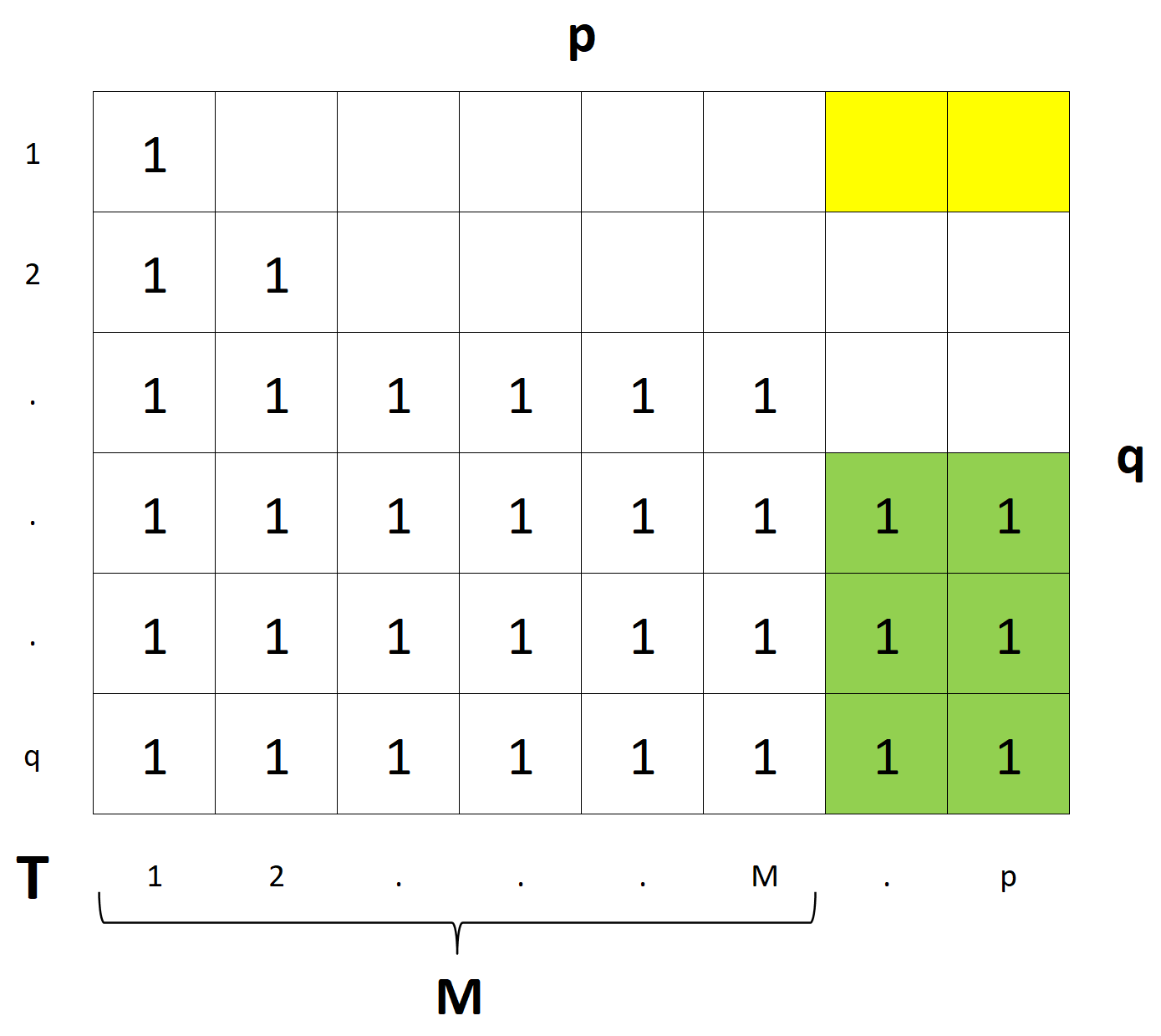}
\hfill\includegraphics[width=0.44\textwidth]{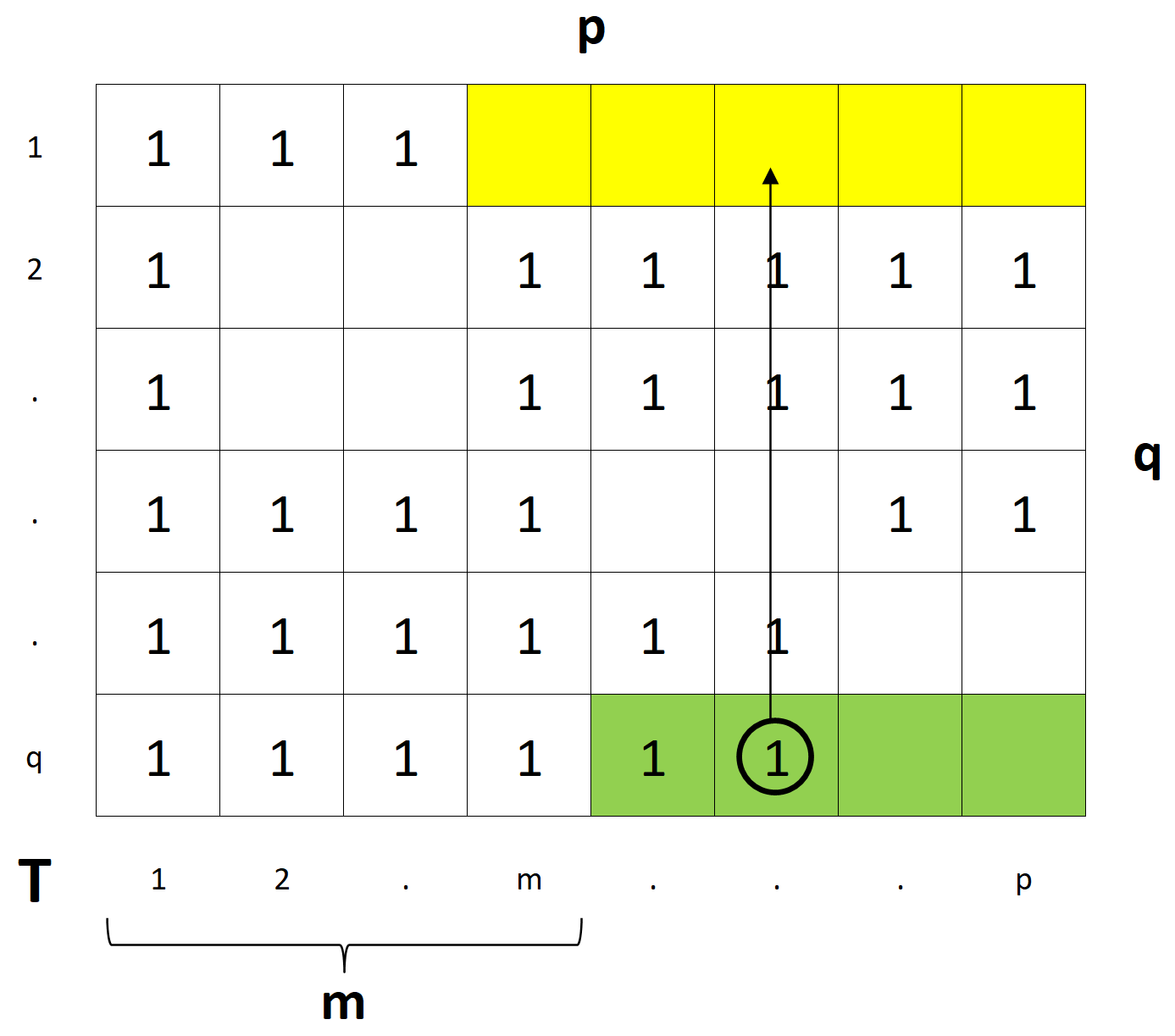}}
\end{tabular}
\caption{}\label{fig1_8}
\end{figure}

\goodbreak

Now we prove a generalization of Ryser’s combinatorial theorem~\ref{theo1.1}.

\begin{theo} \label{theo1.6}
Let $A$ be a rectangle of the size $r\times s$ of a matrix of order $n$. Suppose that each row of $A$ contains at least $s+k-n$ and at most $k$ ones, each column of $A$ contains at least $r+k-n$ and at most $k$ ones. Let $a_0$ be the number of ones in $A$. An illustration is shown in the Figure~\ref{fig1_9}.
If it holds for $a_0$ that
\begin{enumerate}[1)]
\item\label{271} 
$(r+s-n)k \leq a_0$ 
\item\label{272} 
$a_0 \leq rs - (n-k)(r+s-n)$
\end{enumerate}
then the vertical auxiliary brick of $A$, the brick $B$, can be filled with ones such that each column of the brick $A\cup B$ contains exactly $k$ ones, each row contains at least $s+k-n$ and at most $k$ ones, consequently, $A\cup B$ satisfies the conditions of the Ryser’s combinatorial theorem.
\end{theo}

\begin{figure}[htb]
\centering
\includegraphics[width=0.6\textwidth] {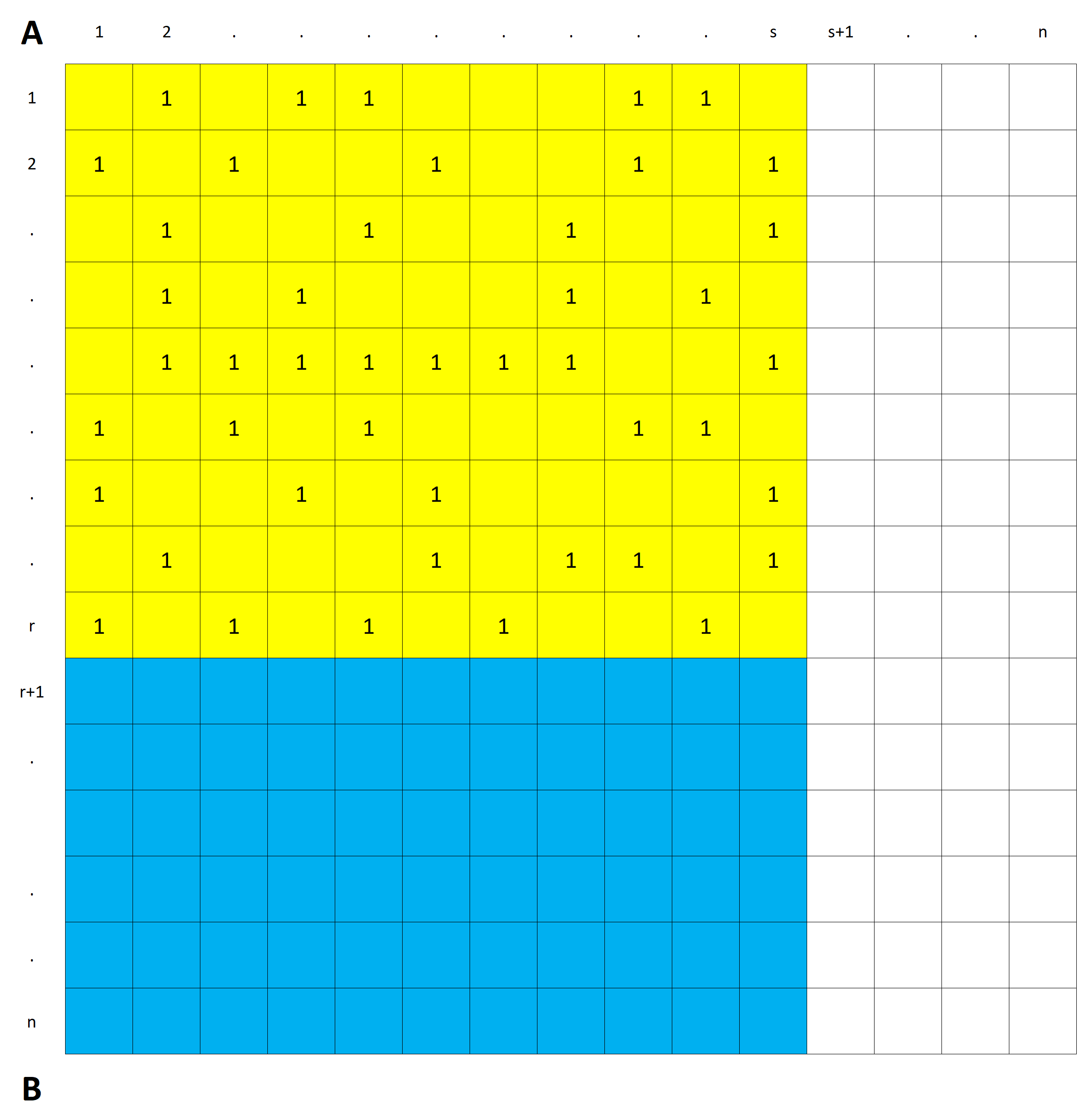}
\caption{}\label{fig1_9}
\end{figure}
\begin{rmrk}\label{rmrk:16nec}
Conditions~\ref{271}) and~\ref{272}) are also \emph{necessary} for such a filling to
exist. Indeed, a filled brick $A\cup B$ contains $sk$ ones, since each of its $s$
columns contains exactly $k$ of them. The brick $B$ must contain at least
$(n-r)(s+k-n)$ ones, otherwise it would have a row containing fewer than $(s+k-n)$
ones, whence
\[
a_0 = sk-\bigl|B\bigr| \leq sk - (n-r)(s+k-n) = rs-(n-k)(r+s-n),
\]
which is condition~\ref{272}). Likewise $B$ can contain at most $(n-r)k$ ones,
otherwise it would have a row containing more than $k$ of them, whence
\[
a_0 = sk-\bigl|B\bigr| \geq sk - (n-r)k = (r+s-n)k,
\]
which is condition~\ref{271}).
\end{rmrk}

\begin{proof}
Let $S(j)$ denote the number of ones in the column $j$ of brick $A$ for each 
$j \in \{1,2,\ldots ,s\}$. Let $c_j = k - S(j)$ for $j \in \{1,2,\ldots,s\}$, then 
\[
N = sk - \sum_{j=1}^sS(j).
\]
We show that for $m = (s + k-n)$, $M = k$, $p = s$ and $q = (n-r)$, the conditions of the Lemma~\ref{lemma1.5} are satisfied, that is, $qm \leq N \leq qM$.
Because of the condition~\ref{271}) of Theorem~\ref{theo1.6}
\[
\sum_{j=1}^sS(j) \geq (r+s-n)k
\] ergo
\[
N = sk - \sum_{j=1}^sS(j) \leq sk - (r+s-n)k = (n-r)k = qM
\]
hence $N \leq qM$.
By the condition~\ref{272}) of Theorem~\ref{theo1.6}
\[a_0 \leq rs-(n-k)(r+s-n).\]
Then because of 
\[
N = sk - \sum_{j=1}^sS(j),
\]
\[  
N \geq sk - [rs - (n-k)(r+s-n)] = (n-r)(s+k-n) = qm,
\]
so $qm \leq N$.
That is, the conditions of the lemma are satisfied, so the brick $B$ can be filled with ones such that each column of the brick $A\cup B$ contains exactly $k$  ones and each row of the brick $A\cup B$ contains at least $s+k-n$ and at most $k$ ones, i.e., $A\cup B$ satisfies the condition of Ryser’s theorem.
\end{proof}
Now we prove the Theorem~\ref{theo1.4}.
\begin{proof}
Suppose that a brick $T$ of size $r\times s\times t$ satisfies the Cruse’s condition system. Rotate the brick $T$ so that $r \geq s \geq t$. We show that the conditions of the Theorem~\ref{theo1.6} are satisfied for the cover sheet of $T$. Let $k = (n-t)$. If there are at least $s+t-n$ rooks in each row layer, then there are at most $s-(s+t-n)=(n-t)=k$ ones in the corresponding row of the cover sheet. If there are at most $t$ rooks in each row layer, then there are at least $(s-t) = s+(n-t)-n = s+k-n$ ones in the corresponding row of the cover sheet. It can also be seen that the corresponding columns of the cover sheet has at most $r-(r+t-n) = (n-t) = k$ and at least $r+(n-t)-n = r+k-n$ ones. 
\begin{figure}[!htb]
\centering
\begin{tabular}{c}
\hbox to .95\textwidth {\includegraphics[width=0.44\textwidth]{CaB_Figures/fig1_10a}
\hfill\includegraphics[width=0.44\textwidth]{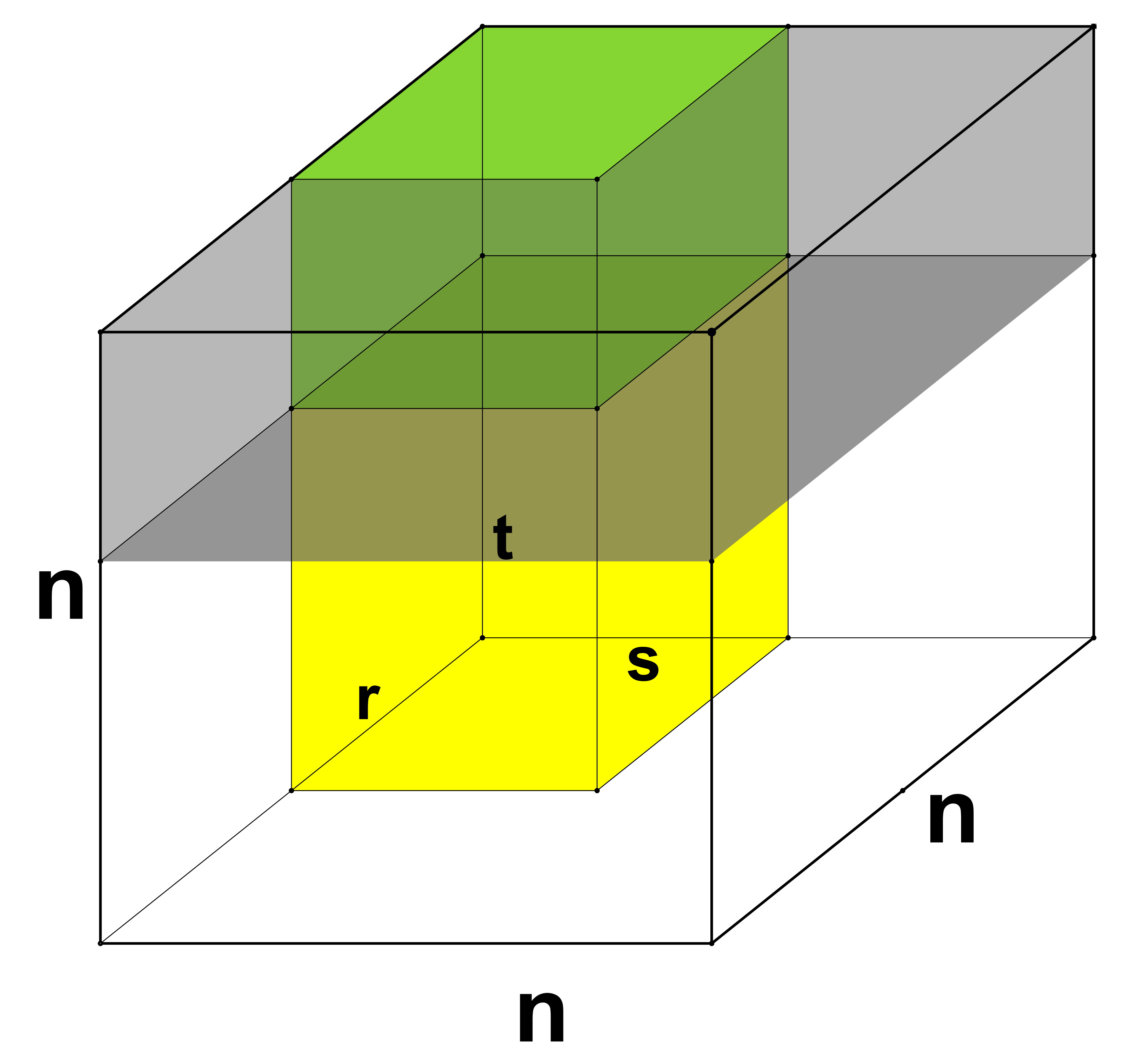}}
\end{tabular}
\caption{}\label{fig1_10}
\end{figure}

The yellow brick $T_0$ contains at most $\capa(n,r,s,t)$ rooks, so the yellow part of the gray cover sheet contains at least 
\[rs - \capa(n,r,s,t) = (r+s-n)(n-t) = (r+s-n)k
\]
ones, ergo condition~\ref{271}) of Theorem~\ref{theo1.6} is satisfied. Each symbol layer of the yellow brick has at least $(r+s-n)$ rooks, so the layers of $T_0$ contain a total of at least $t(r+s-n)$ rooks, so the yellow part of the cover sheet contains at most $rs -t(r+s-n)$ ones, that is, condition~\ref{272}) of Theorem~\ref{theo1.6} is satisfied. The cover sheet of the yellow brick can be filled into a matrix of size $n\times n$, each row and column of which contains exactly $k$ ones. With the help of the resulting matrix, the symbol layers $t+1,t+2,\ldots,n$ can be completed with rooks such that each layer contains $n$ rooks and the rooks are non-attacking. The resulting
$n^2$-brick is denoted by $T_{n-t}$ (grey brick) and a part of it, the vertical auxiliary brick of $T_0$ is denoted by $T_1$ (green brick). We prove that the conditions of the Ryser’s theorem are satisfied for the $n$-brick $(T_0\cup T_1)$ of size $r\times s\times n$. For each symbol layer $\sigma$ of the yellow brick $T_0$, $N(\sigma) \geq r+s-n$ holds because of condition~\ref{253}) of the Cruse' condition system. Since this is a necessary condition for completion of a given symbol layer, it is obviously met for a completed layer and the brick $T_1$ is a part of a completed $n^2$-brick. So, for each symbol layer of the brick $T_1$ is satisfied that $N(\sigma) \geq r+s-n$. Ergo, if you remove all rooks in the $n^2$-brick $T_{n-t}$ outside brick $T_1$, then the PLSC consisting of the $n$-brick $(T_0\cup T_1)$ can be completed because of the Ryser’s theorem.
\end{proof}

\begin{rmrk}
Since $\capa(n,r,s,t) = rs - (n-t)(r+s-n)$, the capacity condition for the yellow brick $T_0$ ensures that the Ryser’s condition can be met for all suitable layers of the green auxiliary bricks of $T_0$. But, $\capa(n,r,s,t) = st - (n-r)(s+t-n)$ and $\capa(n,r,s,t) = rt - (n-s)(r+t-n)$, so the Ryser’s condition can be met for all suitable layers of all auxiliary bricks of $T_0$.
\end{rmrk}

\FloatBarrier

\section{Embedding Compact Bricks}
Let $T_0$ be a brick of size $a\times b\times c$. $T_0$ can be artificially created, that means $T_0$ is not necessarily a brick of a PLSC. Suppose that the rooks in $T_0$ are non-attacking and the files containing rooks do not contain dots.

\begin{defi}
A layer of $T_0$ is called \emph{possible}  if the number of parallel files of the layer containing dots or rooks are identical in both directions of the layer. This number is called the \emph{potential}  of the layer of $T_0$.
\end{defi}

The symbol layer 1 with potential 4 can be seen on the right-hand side of Figure~\ref{fig2_1}. If you place a dot in the cell $(3,3,1)$ then the symbol layer 1 is no longer possible. Let $N_x(i,T_0)$, $N_y(j,T_0)$ and $N_z(\sigma,T_0)$ denote the potentials of the layers of $T_0$, where $i = \{1,2,\ldots ,a\}$, $j = \{1,2,\ldots ,b\}$ and $\sigma = \{1,2,\ldots,c\}$. Easy to check, that all layers of $T_0$ in the left-hand side of Figure~\ref{fig2_1} are possible.

\begin{defi}
The brick $T_0$ is called \emph{compact} if each layer of $T_0$ is possible and
\[
\sum_{i=1}^aN_x(i,T_0) = \sum_{j=1}^bN_y(j,T_0) = \sum_{\sigma=1}^cN_z(\sigma,T_0)
\]
This value is called the \emph{potential} of the brick $T_0$.
\end{defi}

The cases $c = n$, $b = c = n$ and $a = b = c = n$ are allowed. The potential of the compact brick $T_0$ is 15 in the left-hand side of Figure~\ref{fig2_1}.

Let $T_0$ be a compact brick, where 
\[
\sum_{i=1}^aN_x(i,T_0) = \sum_{j=1}^bN_y(j,T_0) = \sum_{\sigma=1}^cN_z(\sigma,T_0) = p_0
\]

\begin{figure}[!htb]
\centering
\begin{tabular}{c}
\hbox to 0.95\textwidth {\includegraphics[width=0.44\textwidth]{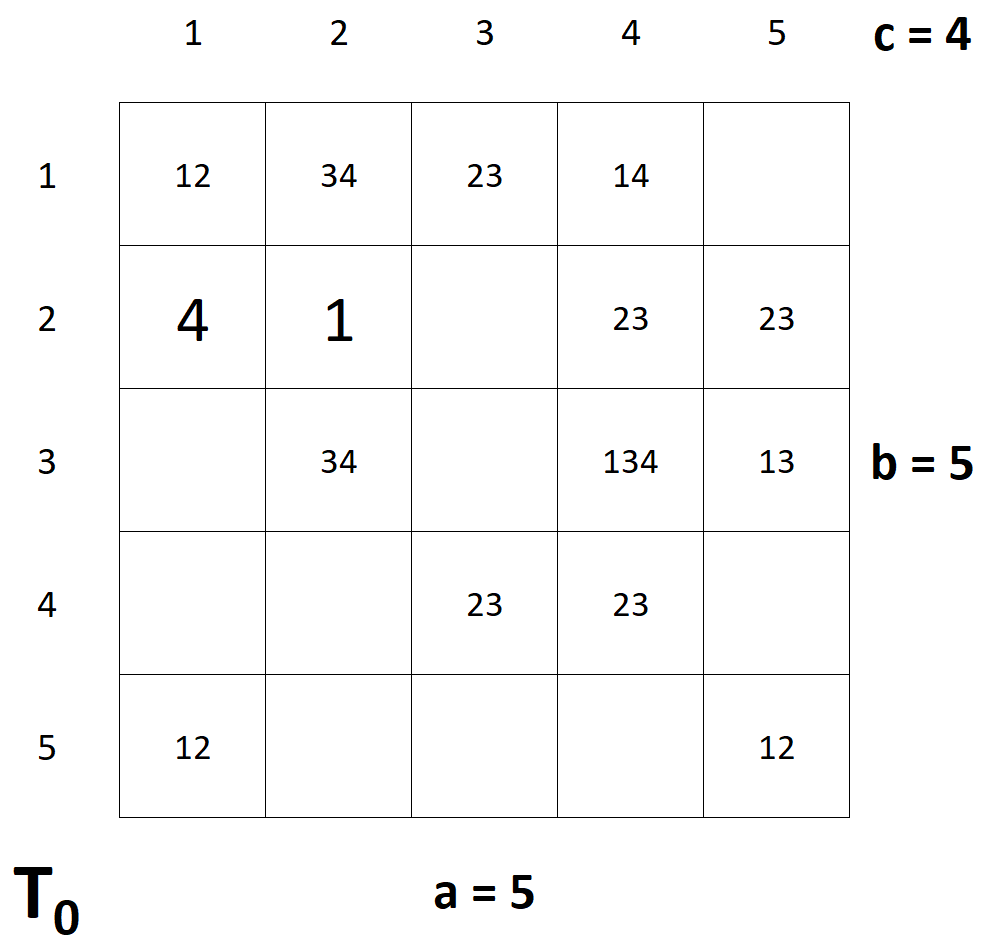}
\hfill\includegraphics[width=0.44\textwidth]{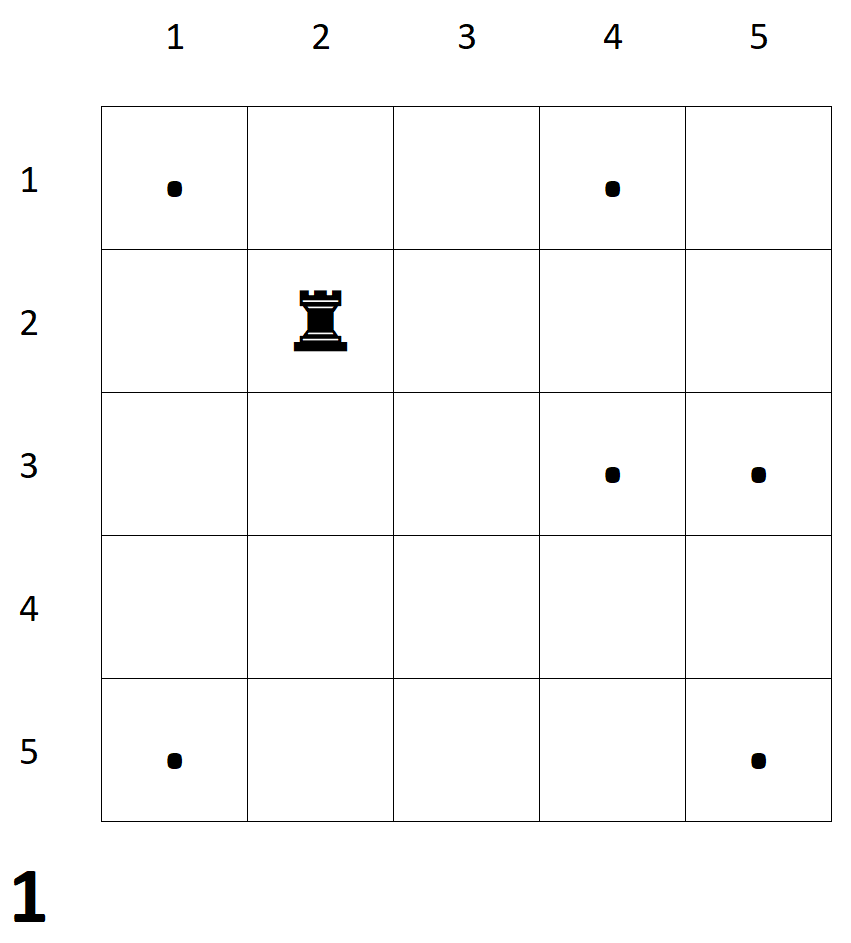}}
\end{tabular}
\caption{}\label{fig2_1}
\end{figure}

\begin{defi}
Let $P$ be an empty PLSC of order $n$, where each cell of $P$ has a dot and let $T_0$ be a compact brick with potential $p_0$ of size $a\times b\times c$, where $a,b,c \leq n$. \emph{Embedding} of $T_0$ in $P$ means that we put $T_0$ into $P$ and in each file of $T_0$ that has rooks or dots we eliminate the dots outside $T_0$ and replace $n^2 - p_0$ dots by rooks outside $T_0$ in such a way that the rooks are non-attacking. The compact brick $T_0$ is \emph{embeddable}  if such a set of rooks exists. This set of rooks is called a \emph{full extension} of $T_0$.
\end{defi}

For example, the compact brick $T_0$ can be the closure hull of the dot structure of \emph{GW6} and the rooks (symbols in the white cells) can be considered as a full extension of the brick $T_0$. $T_0$ is the yellow brick in the Figure~\ref{fig2_2}, its height is 6. The potential of $T_0$ is 12. Outside $T_0$ there are no rooks in the files containing dots in $T_0$ and each file has one rook that has no dots.

\begin{figure}[htb]
\centering
\includegraphics[width=0.6\textwidth] {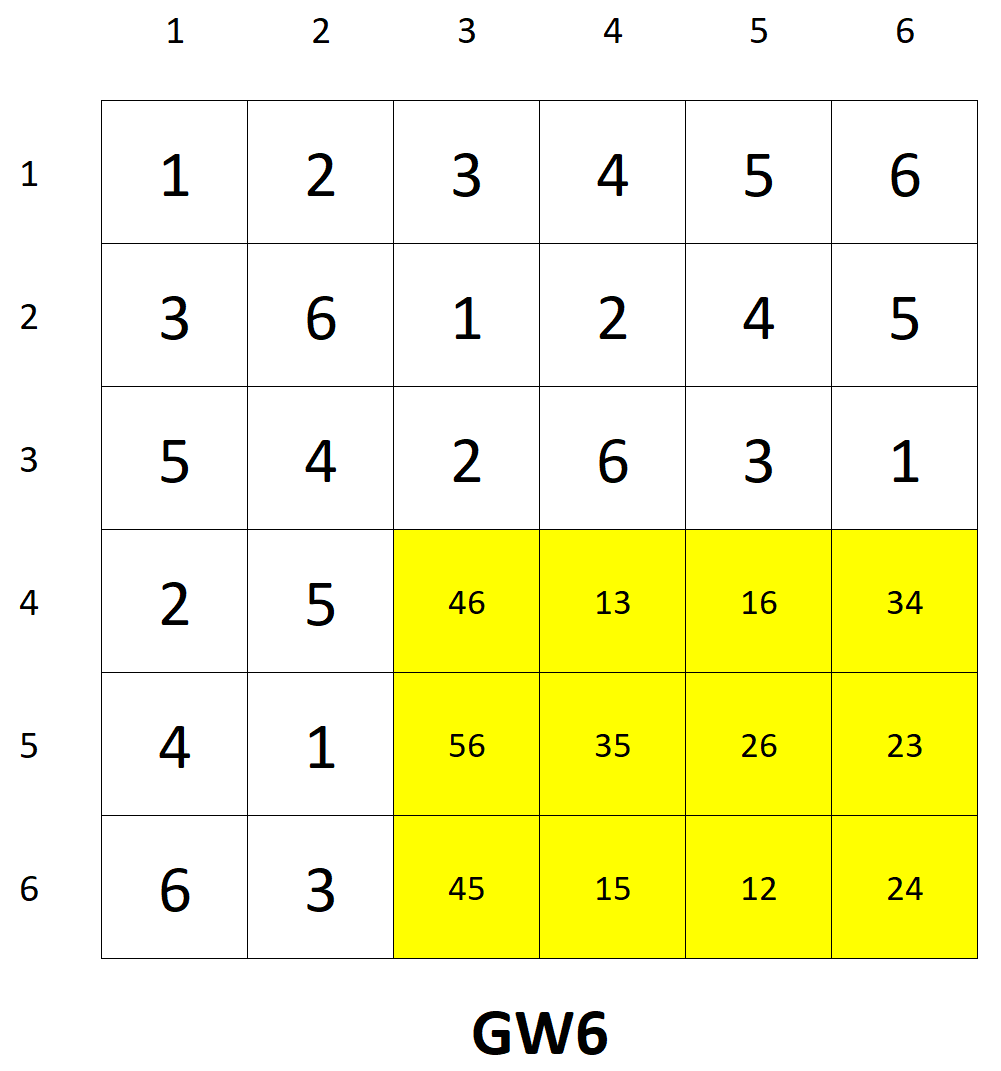}
\caption{}\label{fig2_2}
\end{figure}

\begin{defi} 
A compact brick $T_0$ with potential $p_0$ is called \emph{solvable} or it has a \emph{solution} if $p_0$ dots of $T_0$ can be replaced by rooks in such a way, that the rooks are non-attacking.
\end{defi}

\begin{rmrk} 
Obviously, a solution of $T_0$ and a full extension of $T_0$ combined form a Latin square. 
\end{rmrk}

\begin{rmrk} 
After placing $T_0$ into an empty PLSC and eliminating the proper dots outside $T_0$, the compact brick $T_0$ is “independent of the environment”, i.e., the structure of dots inside $T_0$ has the information whether $T_0$ is solvable or not. It does not have any impact on $T_0$ if you replace a dot with a rook outside $T_0$.
\end{rmrk}

Let $T_0$ be a compact brick of size $a\times b\times c$ of a PLSC of order $n$  with potential $p_0$ and let $N_x(i,T_0)$, $N_y(j,T_0)$ and $N_z(\sigma,T_0)$ denote the potentials of the layers of $T_0$, where $i \in \{1,2,\ldots ,a\}$, $j \in \{1,2,\ldots ,b\}$ and $\sigma \in \{1,2,\ldots,c\}$. 
From the definition of the potential of $T_0$ we have
\[
\sum_{i=1}^aN_x(i,T_0) = \sum_{j=1}^bN_y(j,T_0) = \sum_{\sigma=1}^cN_z(\sigma,T_0) = p_0
\]

\begin{defi} 
The compact brick $T_0$ of size $a\times b\times c$ with potential $p_0$ satisfies the Cruse’s condition system if the followings hold for a given $n$.
\begin{enumerate}
\item 
$N_x(i,T_0) \geq b+c-n$ 	for all $1 \leq i \leq a$
\item 
$N_y(j,T_0) \geq a+c-n$ 	for all $1 \leq j \leq b$
\item 
$N_z(\sigma,T_0) \geq a+b-n$ 	for all $1 \leq \sigma \leq c$
\item 
$p_0 \leq \capa(n,a,b,c)$ 
\end{enumerate}
\end{defi}

\begin{theo} 
Let $T_0$ be a compact brick of size $a\times b\times c$ with potential $p_0$, put $T_0$ into an empty PLSC of order $n$, where $a,b,c \leq n$. If $T_0$ satisfies the Cruse’s condition system, then $T_0$ is embeddable, in other words, there exists a full extension of $T_0$ in the PLSC given.
\end{theo}

\begin{proof}
When you compose the cover sheet of the compact brick $T_0$, you need to treat the files containing rooks or dots in $T_0$ as if they would have a rook inside $T_0$. Then the proof of the Theorem~\ref{theo1.4} also proves this statement.
\end{proof}

It is easy to check, that the closure hull of the dot structure of \emph{GW6}, i.e., the yellow brick in the Figure~\ref{fig2_2}, satisfies the Cruse’s condition system.

Cruse~\cite{[1]} proved that the capacity condition holds for the Cruse’s square and from Corollary~\ref{cor4.21} we know that \emph{GW6} also satisfies the capacity condition, ergo

\begin{cor}
For $n \geq 6$ there exists a PLSC $P$ of order $n$ that contains the closure hull of the dot structure of GW6 in an embedded way and thus satisfies the capacity condition, but nevertheless $P$ is not completable. 
\end{cor}

Figure~\ref{fig2_3} shows examples for $n = 7$, $n = 8$ and $n = 9$.

\begin{figure}[htb]
\centering
\hbox to \textwidth {\includegraphics[width=0.3\textwidth]{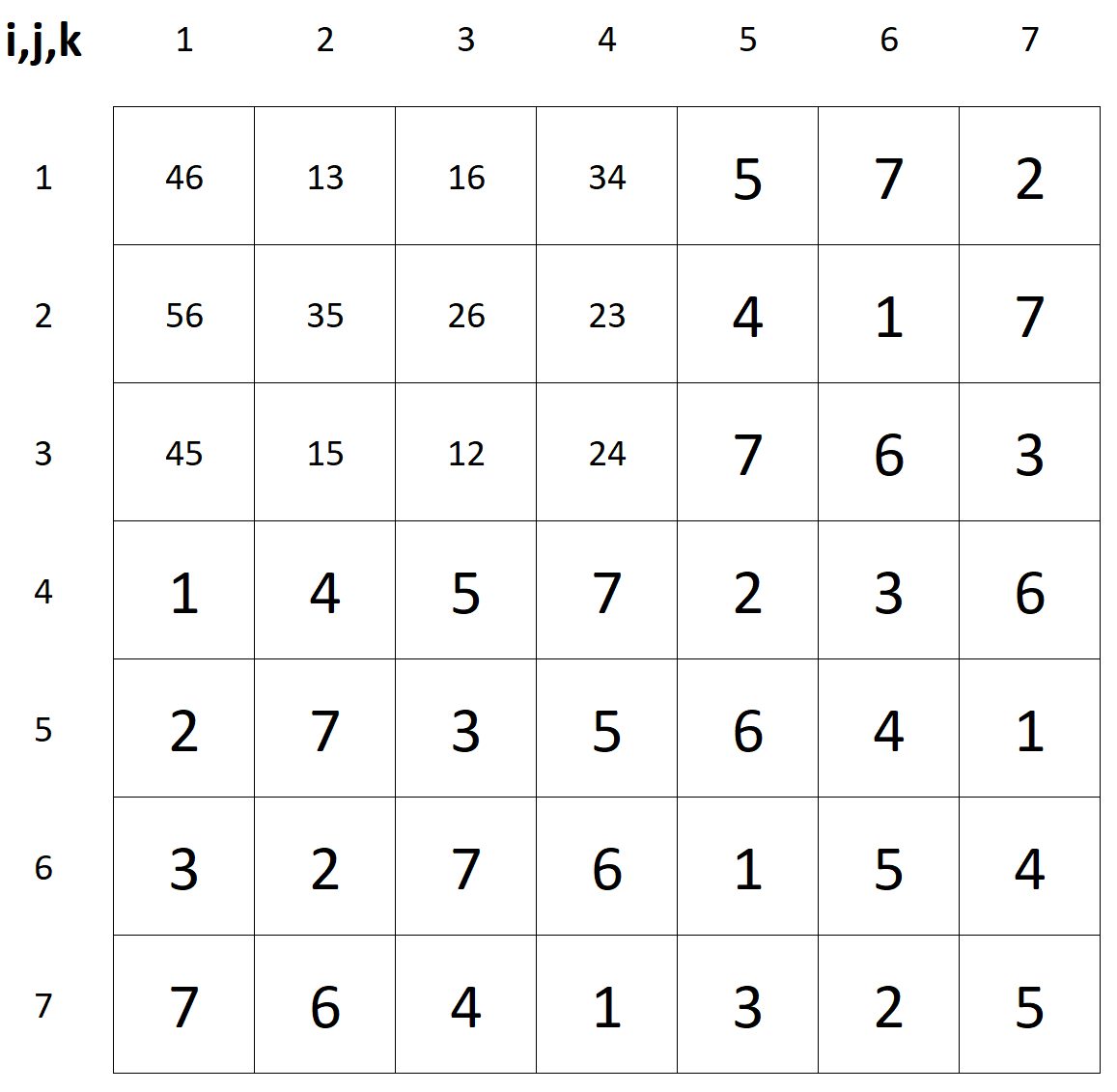}\hfill\includegraphics[width=0.3\textwidth]{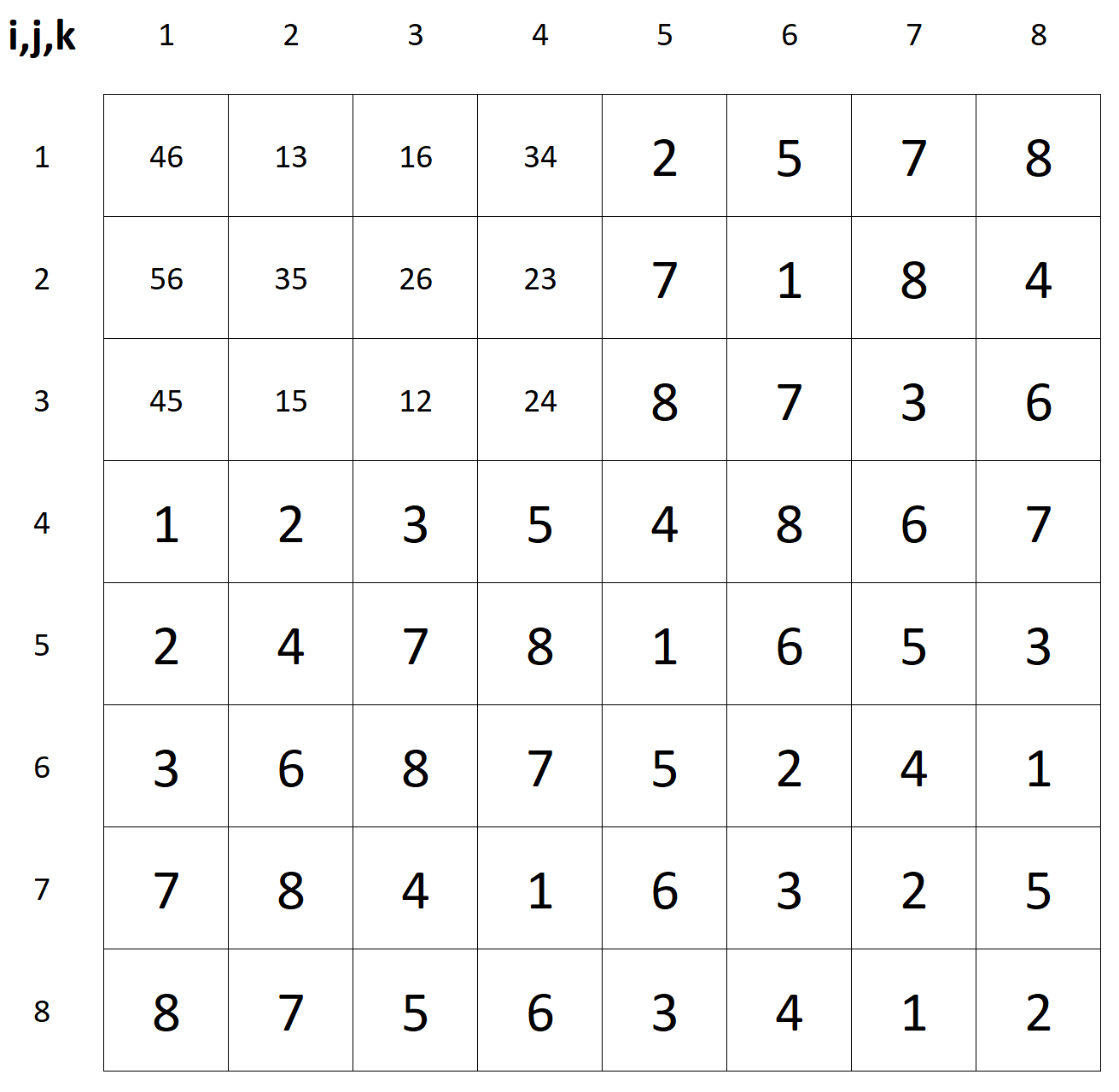}
\hfill\includegraphics[width=0.3\textwidth]{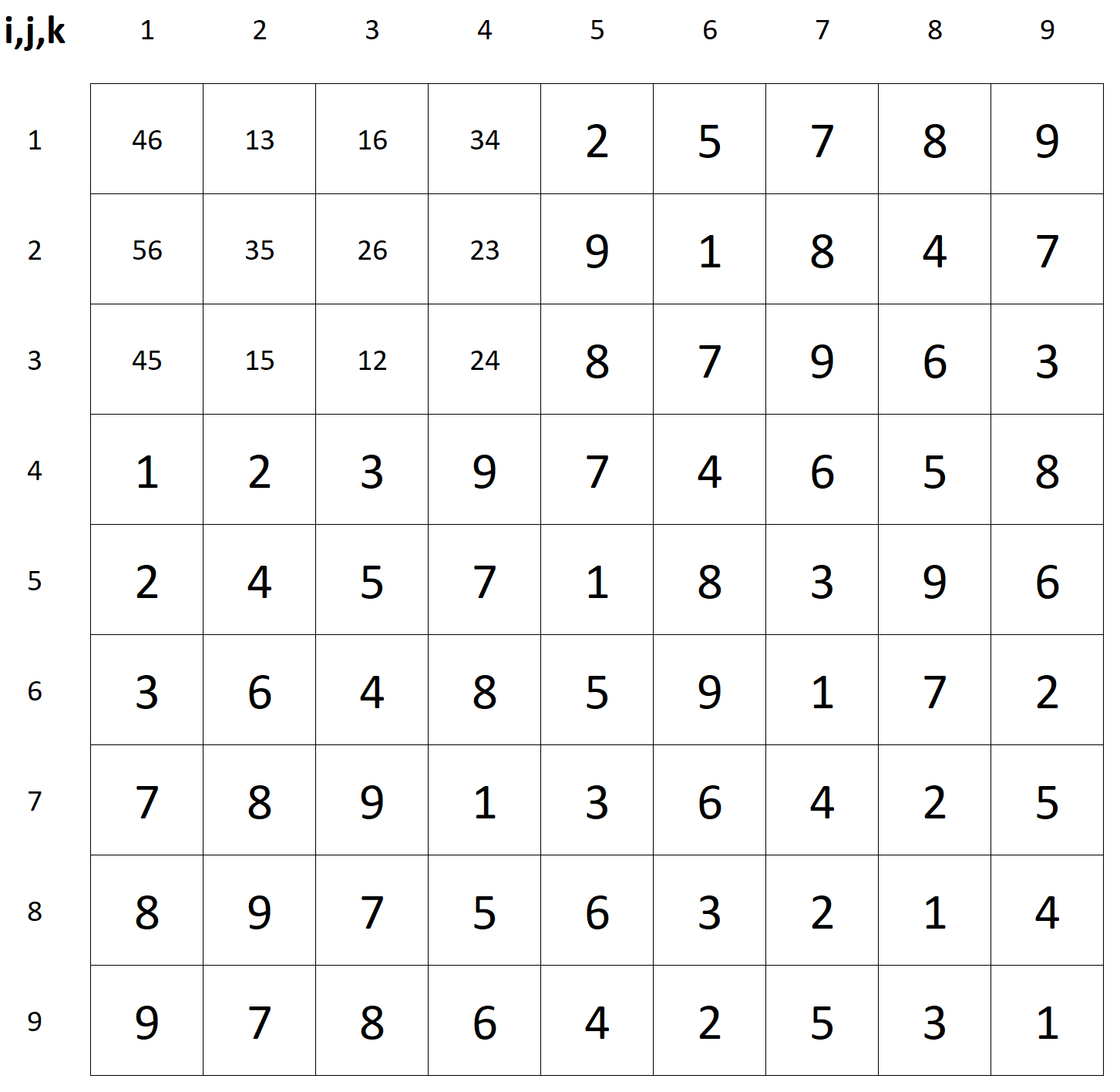}}
\caption{}\label{fig2_3}
\end{figure}

\begin{rmrk}
It is easy to prove that for each compact brick $T_0$ there exists $N$, for example $N = \vcount(T_0)$, such that $T_0$ satisfies the Cruse's condition system for all $n > N$.
\end{rmrk}

\FloatBarrier

\section{Elimination, BUGs}
The terms PLS and the corresponding PLSC, candidate and the corresponding dot, symbol and the corresponding rook are used in parallel, but the context should always make it clear which one we are thinking of. The examples normally are in PLS form, but they are always a projection of the current PLSC. If it improves clarity, we show all three relevant projections of the PLSC (primary conjugates). We always assume that the automatic elimination has already been performed, hence, the files containing rooks do not contain dots.

\begin{defi} 
Let P be a PLSC. A dot of P is called \emph{desolate} if there exists a file containing only the given dot. The cell is called \emph{desolate} if it has a desolate dot. To \emph{treat} a desolate cell means to replace the dot in the given cell with a rook and to eliminate the dots that are in the Hamming sphere of radius 1 whose center is the given desolate cell.
\end{defi}

\goodbreak
Let P be a PLSC. 
\begin{algo} [Primary Extension] \label{algo3.2}
\[\]
\begin{enumerate} [{Step }1{:}]
\item 
Let $j = 0$ and $P_0 = P$
\item 
If each file of $P_j$ contains one rook, then stop ($P_j$ is completed)
\item 
If there is an eliminated file in $P_j$, then stop ($P_j$ is not completable)
\item 
If there are no desolate cells, then stop
\item 
Increment $j$ by one, arbitrarily select a desolate cell, denote it by $D_j$ and replace the dot in this cell by a rook
\item 
Perform automatic elimination, denote the result by $P_j$, and continue from Step 2
\end{enumerate}
\end{algo}
We eliminate dots if and only if there is a rook in one of the files containing the given dot, and we replace a dot with a rook if and only if the dot is desolate. If we stopped because there is an eliminated file, then $P$ is not completable, and we do not define primary extensions for such a PLSC. In other cases, the result of the procedure will be denoted by $P^*$ ($P^*=P_0=P$ if $P$ has no desolate dots at all). $P^*$ is either an LSC or $P^*$ has at least two dots in each file that contains no rooks. It is obvious that $P \subseteq P^*$. We prove, that $P^*$ is unique, so the PLSC $P^*$ can be considered \textbf{the} primary extension of $P$.

Suppose that there is at least one desolate dot in $P$. Let $A_1$ and $A_2$ be two runs of the procedure~\ref{algo3.2}, and the result of $A_1$ is $R_1$, the result of $A_2$ is $R_2$. The run $A_1$ defines a cell sequence $D_1,D_2,\ldots ,D_k$ and the run $A_2$ defines a cell sequence $E_1,E_2,\ldots ,E_m$. The procedure starts with a desolate cell, so the cells $D_1$ and $E_1$ are desolate. Both sequences contain all the cells in that order, in which the procedure treated them. When the procedure is completed, the treated cells contain non-attacking rooks. 
\begin{cor} \label{cor3.3}
The Hamming distance of any two cells of the sequence $D_1,D_2,\ldots,D_k$ or of the sequence $E_1,E_2,\ldots,E_m$ is at least 2.
\end{cor}
\begin{statm}
The runs $A_1$ and $A_2$ produce the same result in the following sense. If $R_1$ contains an eliminated file, then $R_2$ also contains an eliminated file. If $R_1$ is an LSC (completed PLSC), then $R_2$ also is an LSC and $R_1 = R_2$. Otherwise, $R_1 = R_2$, including both the rook and dot structures.
\end{statm}
\begin{proof} 
Let $S_1(D_j)$ denote the Hamming sphere of radius 1 whose center is $D_j$, where $j \in {1,2,\dots ,k}$. Let $F(j) = S_1(D_1) \cup S_1(D_2) \cup \ldots \cup S_1(D_j)$, where $j \in \{1,2,\ldots ,k\}$. If we arbitrarily pick a cell from the sequence $D_1,D_2,\ldots ,D_k$ in $P$ (prior to any run of the procedure) and treat it then we eliminate all dots in $S_1(D_j)$. Due to Corollary~\ref{cor3.3} and because of $S_0(D_j) = \{D_j\}$, we do not eliminate dots in any cell of the sequence $D_1,D_2,\ldots ,D_k$. Therefore, the elimination in the set $F(j)$ can be done in any order for any $j \in \{1,2, \ldots ,k\}$. If the result of the run $A_1$ is that there is an eliminated file, then the result of the run $A_2$ is that there is an eliminated file. Otherwise, run $A_2$ stops because there are no more desolate dots. If a desolate dot is eliminated, an eliminated file is created. So, cell $D_1$ has not been eliminated, but then cell $D_1$ occurs in run $A_2$ and is treated, so cell $D_2$ becomes desolate. If it has been eliminated, then there is an eliminated file if not, then it is treated and cell $D_3$ becomes desolate, etc... If all cells in the sequence $D_1,D_2,\ldots ,D_k$ have been treated, then we know from the run $A_1$, that there is an eliminated file. So, the run $A_2$ finds eliminated file, which is a contradiction. 
The roles of runs $A_1$ and $A_2$ are symmetric, so if run $A_1$ did not stop because of an eliminated file, then run $A_2$ do not stop because of an eliminated file. Suppose that neither $A_1$ nor $A_2$ found an eliminated file. Then we do not eliminate the dot in cell $E_1$ in run $A_1$, otherwise an eliminated file is created. This means, that cell $E_1$ appears in the cell sequence of run $A_1$ at the latest when there is no other desolate cell. Let $D_{j_1} = E_1$. Similarly, we do not eliminate the dot in cell $E_2$ in run $A_1$, otherwise an eliminated file is created. So, cell $E_2$ appears in the cell sequence of run $A_1$ at the latest when the cell $D_{j_1}=E_1$ is treated, because we know from the run $A_2$ that $E_2$ is desolate if $E_1$ is treated. Let $D_{j_2} = E_2$. We do not eliminate the dot in cell $E_3$ in run $A_1$, otherwise an eliminated file is created. So, cell $E_3$ appears in the cell sequence of run $A_1$ at the latest when the cell $D_{j_1}=E_1$ and $D_{j_2}=E_2$ are treated, because we know from the run $A_2$ that $E_3$ is desolate if $E_1$ and $E_2$ are treated in any order. And so on. Consequently, any cells treated in run $A_2$ occurs in run $A_1$. The roles of $A_1$ and $A_2$ are symmetric, so two runs result in the same rook structure, which also implies the same dot structure due to elimination of the dots in cells of $F(k)$. Ergo, $P^*$ is unique.
\end{proof}
\begin{cor} 
An LSC $Q$ is a completion of $P$ if and only if it is a completion of $P^*$. Thus, if $P^*$ is an LSC, then $P$ has exactly one completion, namely $P^*$.
\end{cor}
Although in the primary extension some dots were replaced by rooks and some dots were eliminated, the result may still contain more dots that cannot become rooks in any extension. 
\begin{defi} 
Let P be a PLSC. A dot of $P$ is \emph{deceptive} if at least one of the next two conditions holds for the given dot:
\begin{enumerate} [1)]
\item
$(T_0,T_3)$ is a stuffed RBC of $P$ and the dot is either in $T_0$ or in $T_3$
\item
$(T_0,T_2)$ is a balanced RBC of a layer of $P$ and one brick of the RBC is perfect and the dot is in the other brick
\end{enumerate}
\end{defi}
\begin{rmrk}
If we replace a deceptive dot by a rook, then we create an overloaded RBC or an underweighted perfect brick, thus the PLSC we get is not completable. 
\end{rmrk}
\begin{rmrk}
We define the same set of deceptive dots in any order we find them, so the set of deceptive dots of a PLSC is unique, so that the deceptive dots can also be eliminated in any order.
\end{rmrk}

Let P be a PLSC. Suppose that the primary extension of $P$ exists. Denote it by $P^*$.

\goodbreak

\begin{algo}[Secondary extension] \label{algo3.7} 
\[\]
\begin{enumerate} [{Step }1{:}]
\item
Let $j = 0$ and $P_0^* = P^*$
\item
If the capacity condition does not hold for $P_j^*$, then stop.
\item
If there are no deceptive dots in the PLSC, then stop.
\item
Eliminate the deceptive dots and perform primary extension
\item
If primary extension created an eliminated file, then stop
\item
Increment $j$ by one, denote the result by $P_j^*$ and continue from Step 2
\end{enumerate}
\end{algo}
If we stopped in Step 2 or Step 5, then $P$ is not completable, and we do not define secondary extensions for such a PLSC. In other cases, let the result of the procedure is denoted by $P^{**}$ ($P^{**}=P_0^*=P^*$ if $P^*$ has no deceptive dots at all). $P^{**}$ is either an LSC or has at least two dots in each file that contains no rooks.
The primary extensions $P_j^*$ are unique and
\[
P \subseteq P_1^* \subseteq P_2^* \subseteq \ldots \subseteq P_j^* \subseteq \ldots  \subseteq P^{**}
\]
The set of deceptive dots also are unique, so $P^{**}$ is unique, thus $P^{**}$ can be considered \textbf{the} secondary extension of $P$. 

\begin{cor}
An LSC $Q$ is a completion of $P$ if and only if it is a completion of $P^{**}$. Hence, if $P^{**}$ is an LSC, then $P$ has exactly one completion, namely $P^{**}$.
\end{cor}
We apply the primary and secondary extension algorithms on the Cruse’s square and on its conjugates. The Cruse’s square has no desolate dots, so its primary extension is itself, as it can be checked in the Figure~\ref{fig3_1}.

\begin{figure}[htb]
\centering
\begin{tabular}{c}
\hbox to .95\textwidth {\includegraphics[width=0.3\textwidth]{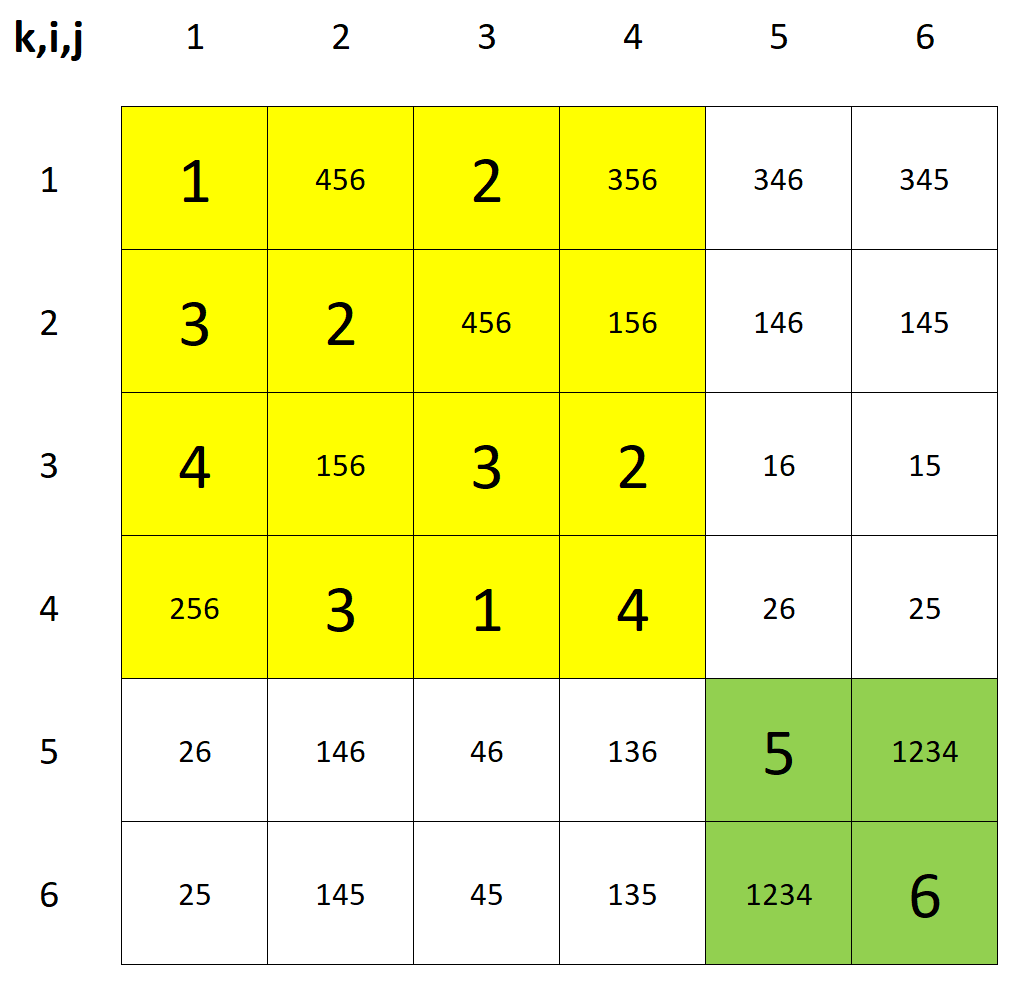}\hfill\includegraphics[width=0.3\textwidth]{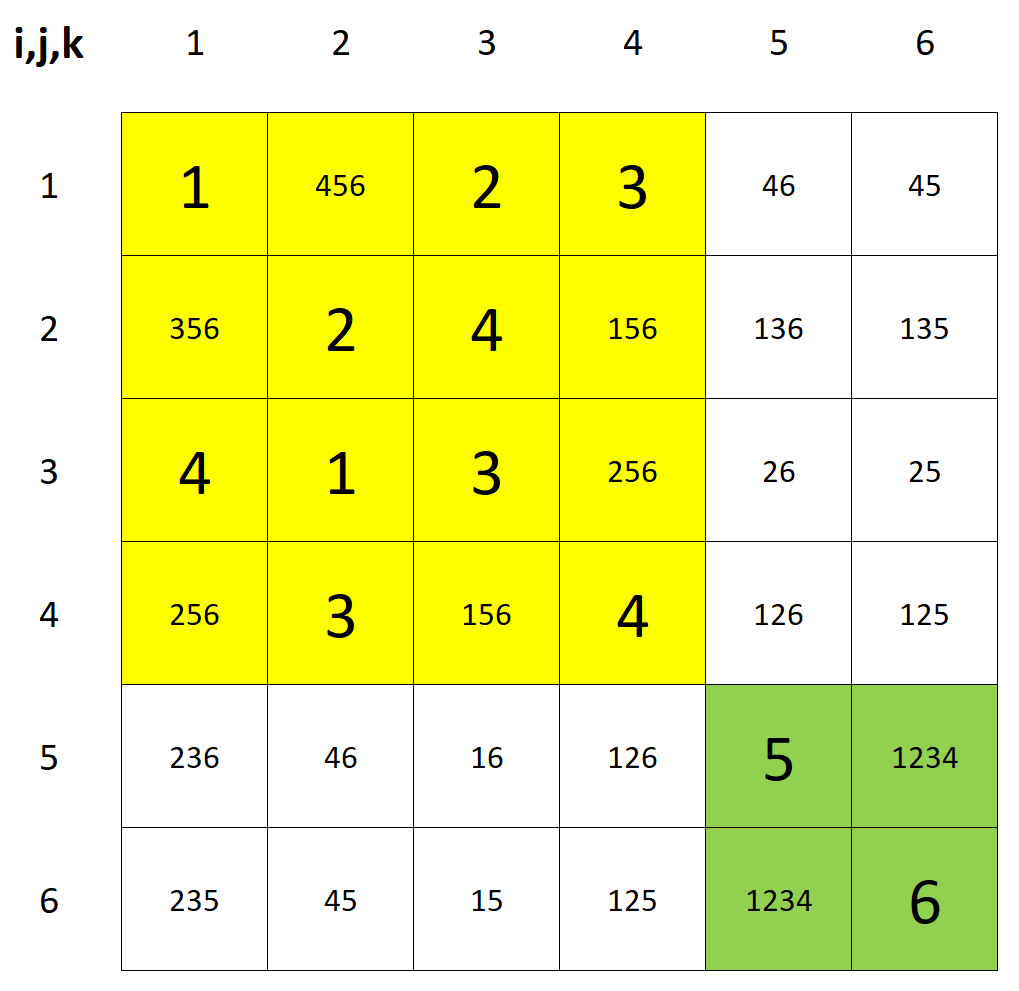}
\hfill\includegraphics[width=0.3\textwidth]{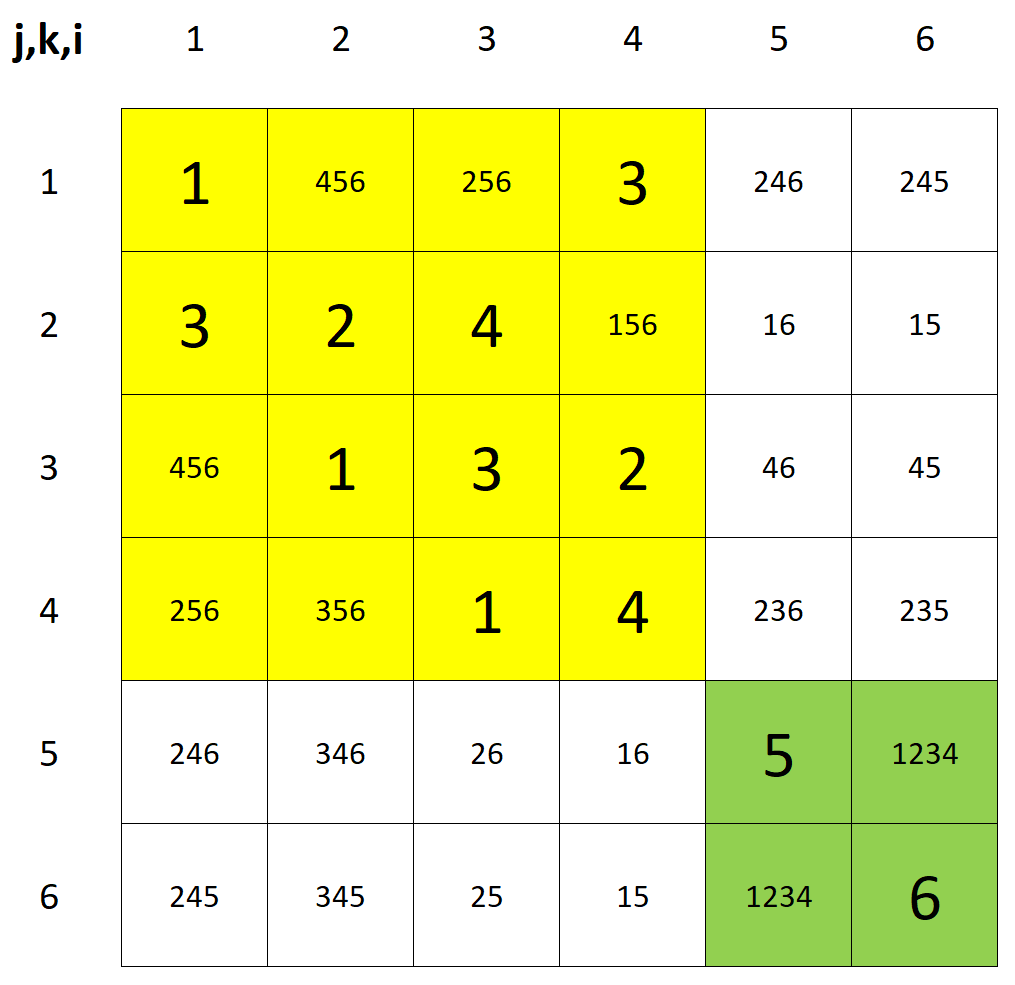}}
\end{tabular}
\caption{}\label{fig3_1}
\end{figure}

Since no new rooks are created and as we will see later in Statement~\ref{statm4.19} and Corollary~\ref{cor4.21}, the capacity condition holds for Cruse’s square, ergo the capacity condition is also satisfied after each elimination step. Let $T_0$ be the yellow brick of size $4\times 4\times 4$, let the remote mate of $T_0$ is the green brick $T_3$ of size $2\times 2\times 2$. The capacity of the RBC $(T_0,T_3)$ is $\dfrac{4\cdot 4\cdot 4+2\cdot 2\cdot2 }{6} = 12$. The RBC contains 12 rooks, so it is stuffed. Thus, we can eliminate the candidates 1,2,3 and 4 in the empty cells of the yellow brick. So, in the empty cells of the yellow brick only 5 and 6 can occur as candidates. In the same way, we can eliminate the candidates 5 and 6 in the empty cells of the green brick (in our case there are no such candidates), so only 1, 2, 3 and 4 can appear as candidates in the empty cells of the green square. The eliminated state is shown in the Figure~\ref{fig3_2}.

\begin{figure}[htb]
\centering
\begin{tabular}{c}
\hbox to .95\textwidth {\includegraphics[width=0.3\textwidth]{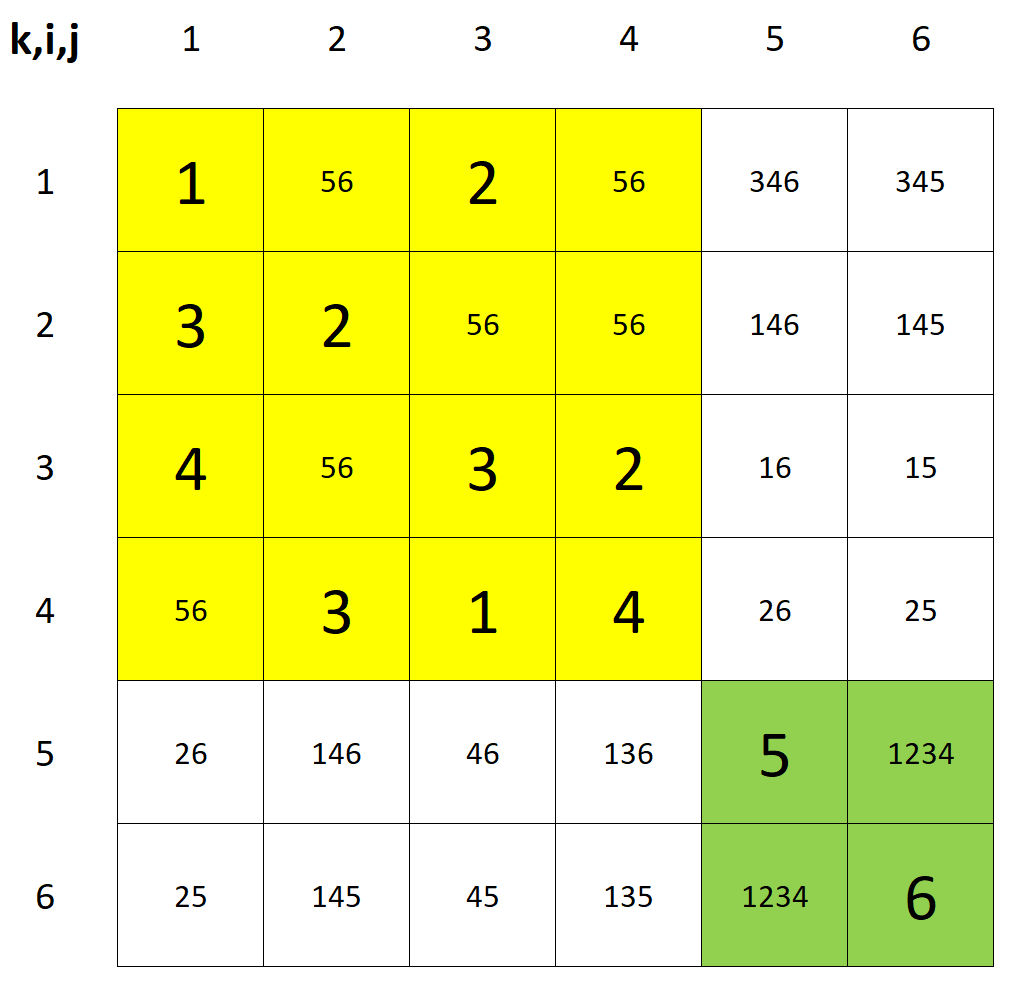}\hfill\includegraphics[width=0.3\textwidth]{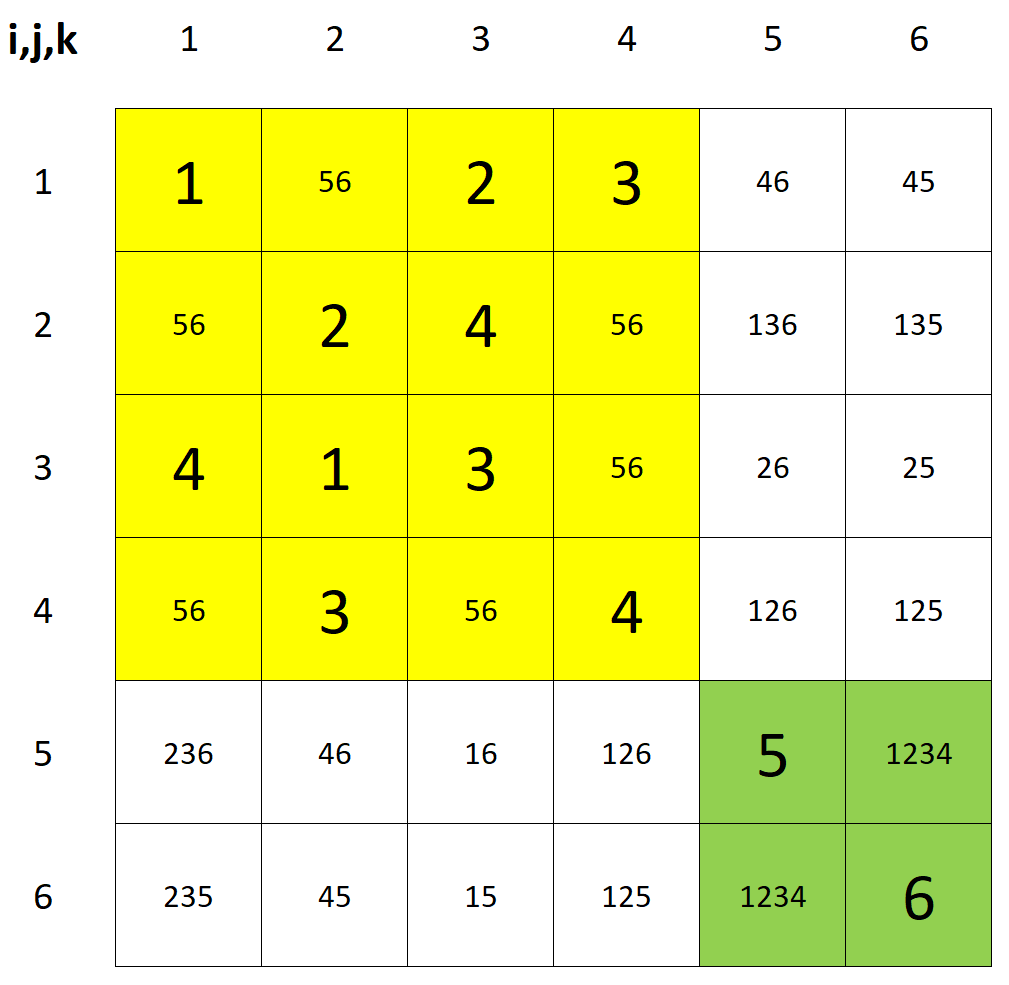}
\hfill\includegraphics[width=0.3\textwidth]{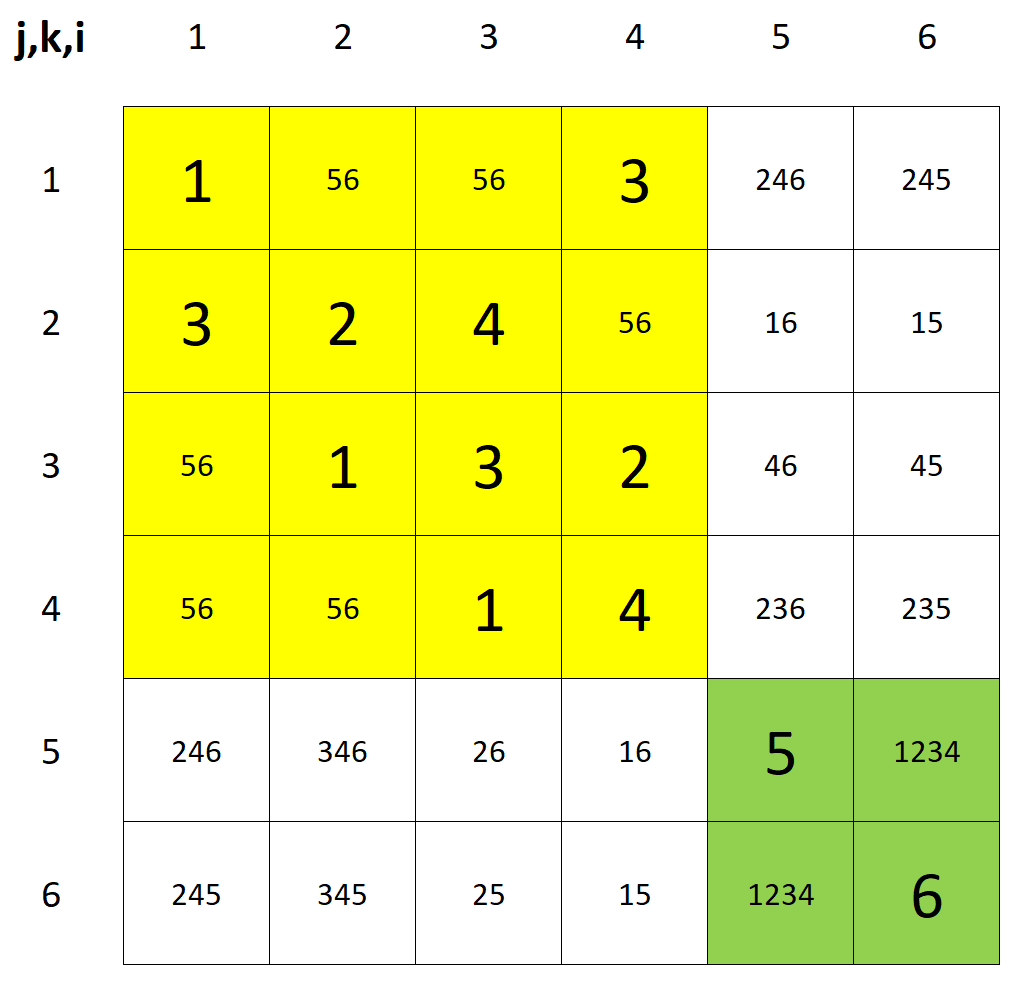}}
\end{tabular}
\caption{}\label{fig3_2}
\end{figure}

There are no dots in the yellow brick, so every layer of the brick is perfect. If $T$ is a layer of the yellow brick and contains 4+4-6 = 2 rooks, then the candidates can be eliminated in the remote mate of $T$ in the given layer. Layer 1 is shown in the middle of Figure~\ref{fig3_3}. The yellow brick in layer 1 has two rooks and its Ryser-number is 2, so candidates 1 and the corresponding dots in the part of the layer below the green brick can be eliminated.

\begin{figure}[htb]
\centering
\begin{tabular}{c}
\hbox to .95\textwidth {\includegraphics[width=0.3\textwidth]{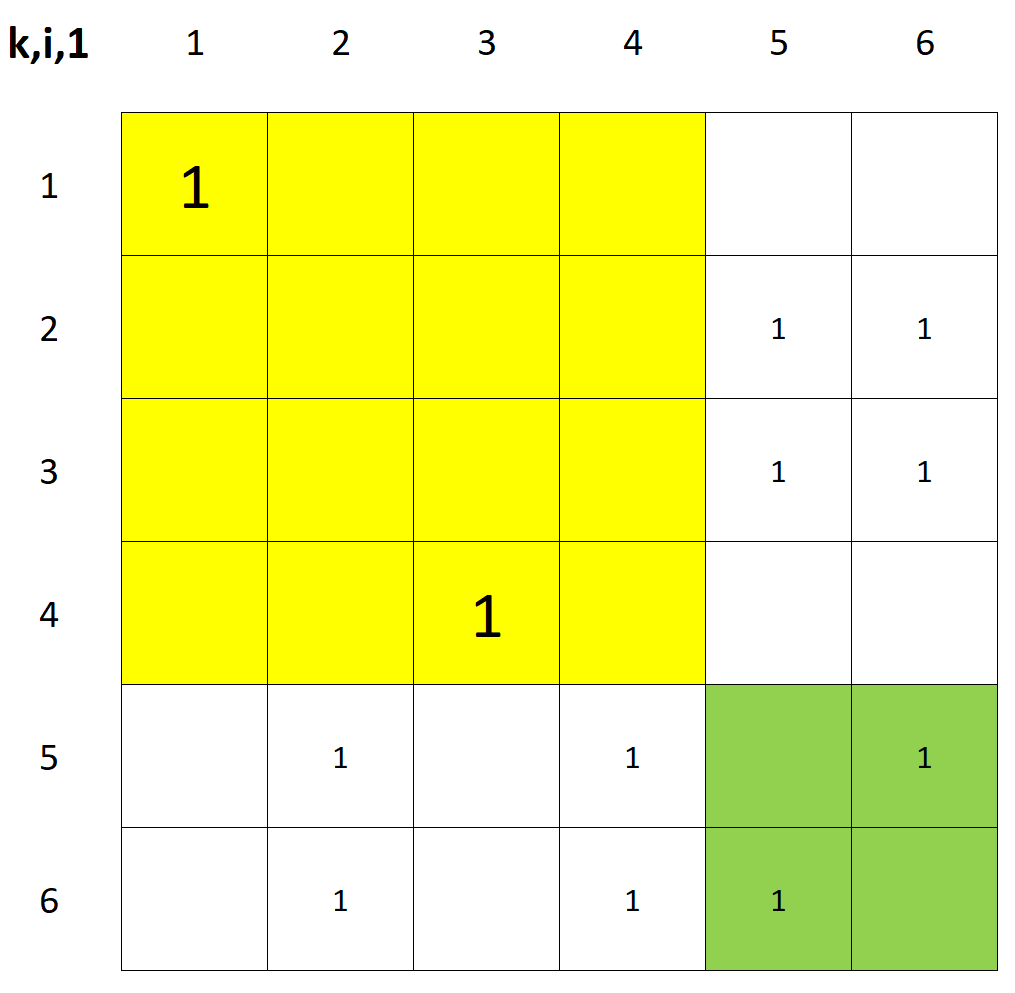}\hfill\includegraphics[width=0.3\textwidth]{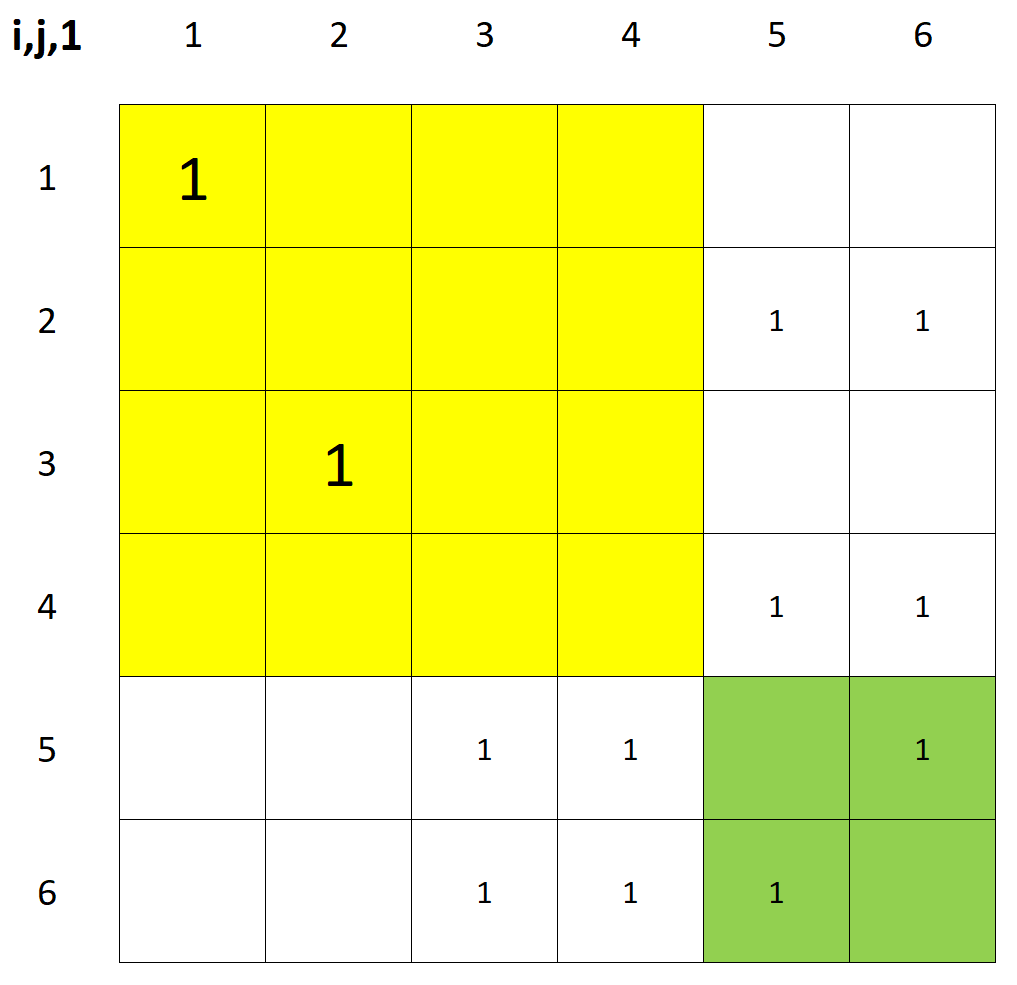}
\hfill\includegraphics[width=0.3\textwidth]{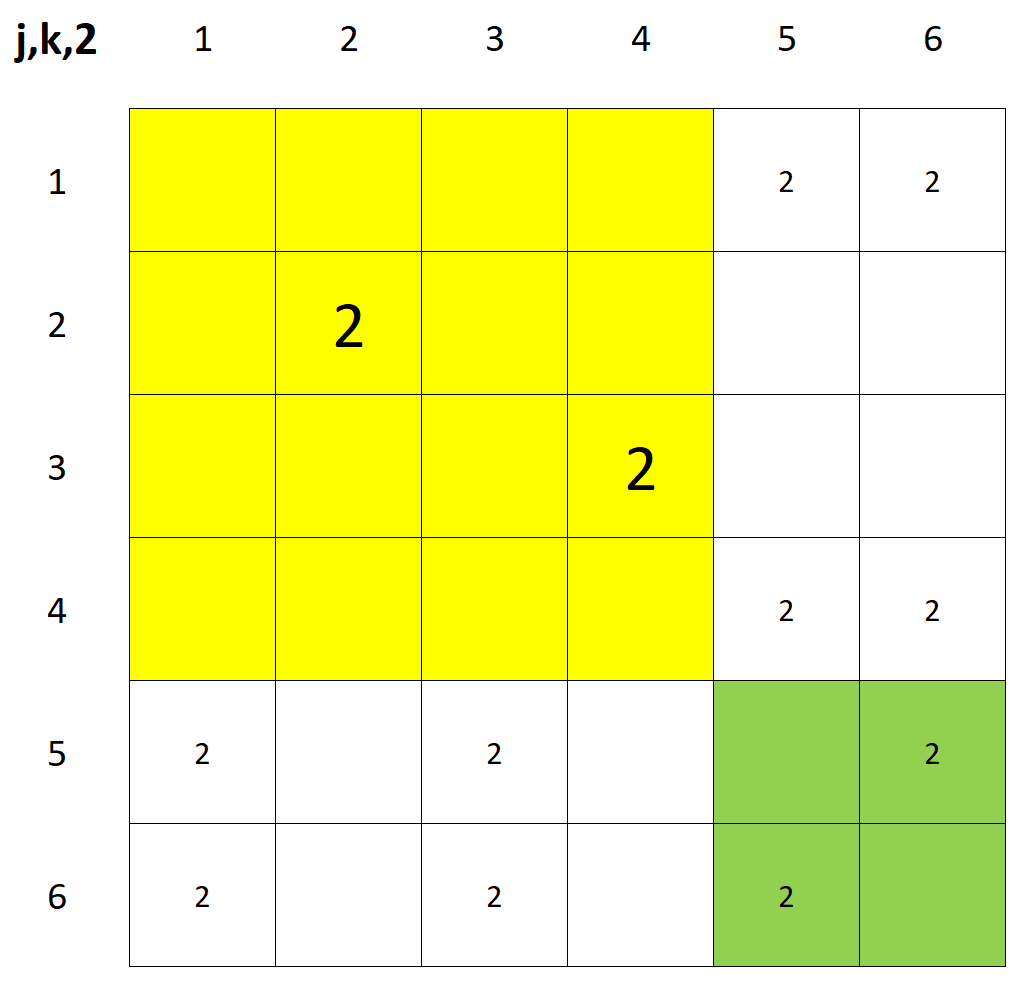}}
\end{tabular}
\caption{}\label{fig3_3}
\end{figure}

The column layer 1 shown in the middle of Figure~\ref{fig3_2} also contains two rooks. The layer can be seen on the left-hand side of Figure~\ref{fig3_3}. The two candidates that can be eliminated are the dots with coordinates $(5,6,1)$ and $(6,5,1)$. These dots are the candidates with coordinates $(6,1,5)$ and $(5,1,6)$ shown in the middle of Figure~\ref{fig3_2}. The row layer 2, shown in the middle of Figure~\ref{fig3_2} also contains two rooks. The layer can be seen on the right-hand side of Figure~\ref{fig3_3}. The two candidates that can be eliminated are the dots with coordinates $(5,6,2)$ and $(6,5,2)$. These dots are the candidates with coordinates $(2,5,6)$ and $(2,6,5)$ shown in the middle of Figure~\ref{fig3_2}. If we perform all the eliminations allowed by the perfect yellow layers with two rooks, we obtain the following squares in the Figure~\ref{fig3_4}.

\begin{figure}[htb]
\centering
\begin{tabular}{c}
\hbox to .95\textwidth {\includegraphics[width=0.3\textwidth]{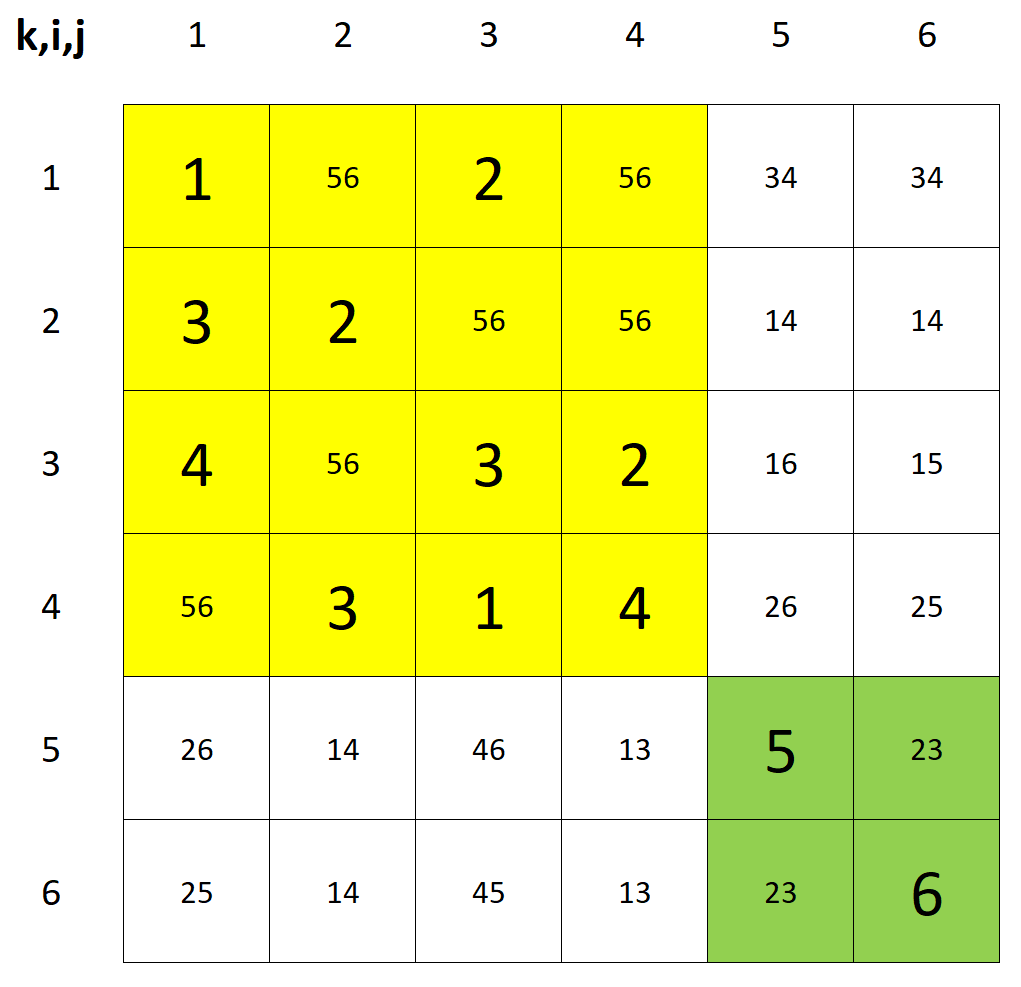}\hfill\includegraphics[width=0.3\textwidth]{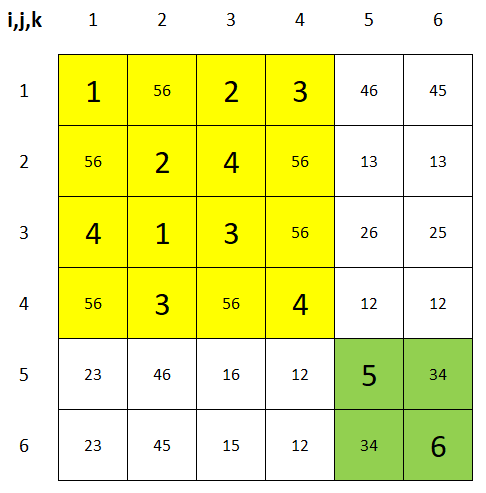}
\hfill\includegraphics[width=0.3\textwidth]{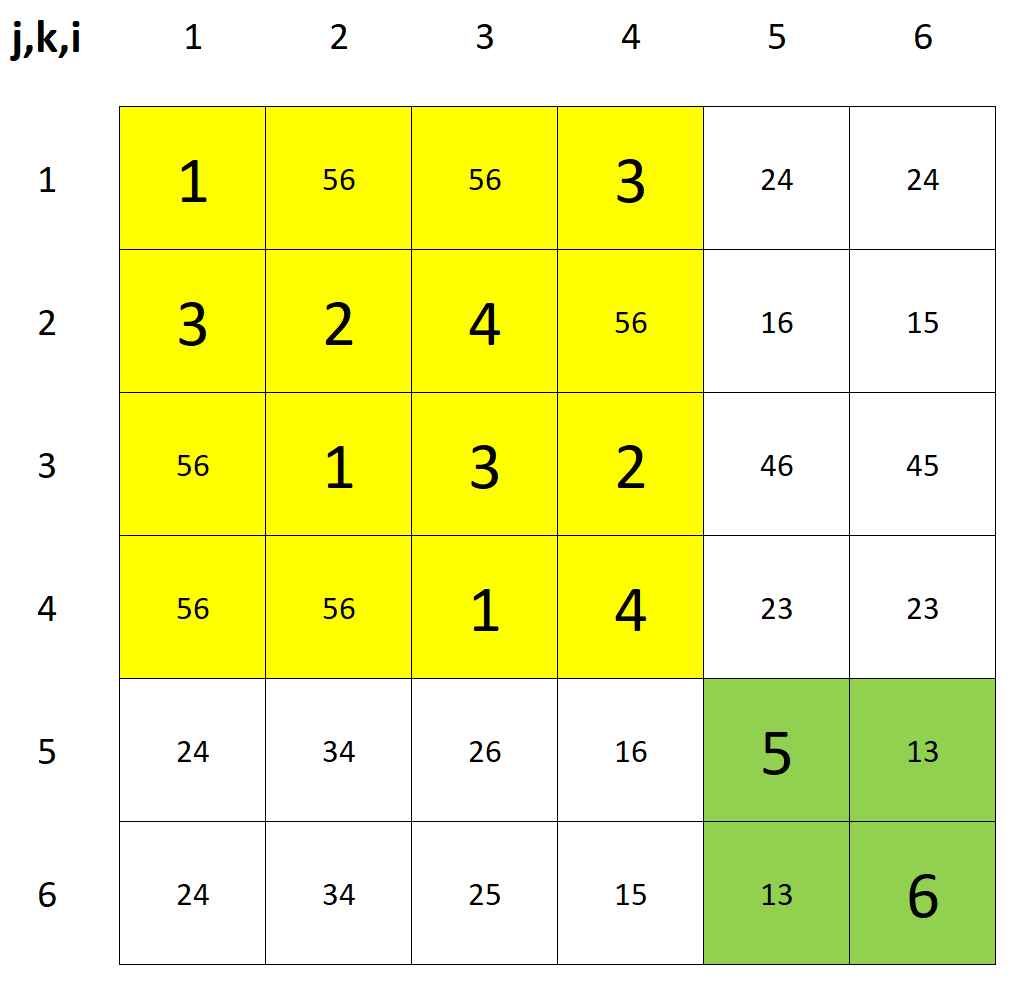}}
\end{tabular}
\caption{}\label{fig3_4}
\end{figure}

\begin{defi} 
A structure of dots in a PLSC is called \emph{minimum candidate structure}, or mCS  for short if each file of the PLSC that contains no rooks, contains exactly two dots.
\end{defi}
\begin{rmrk} 
The number of dots in an mCS is even, twice the number of empty cells. 
Consequently, if the dot structure of a PLSC is an mCS, then each conjugate of the proper PLS has the same number of empty cells.
\end{rmrk}
\begin{rmrk}
The Latin square \emph{GW6} has an mCS, since all empty cells of the three primary conjugates contain exactly two candidates as depicted in the Figure~\ref{fig3_5}.
\end{rmrk}

\begin{figure}[htb]
\centering
\begin{tabular}{c}
\hbox to .95\textwidth {\includegraphics[width=0.3\textwidth]{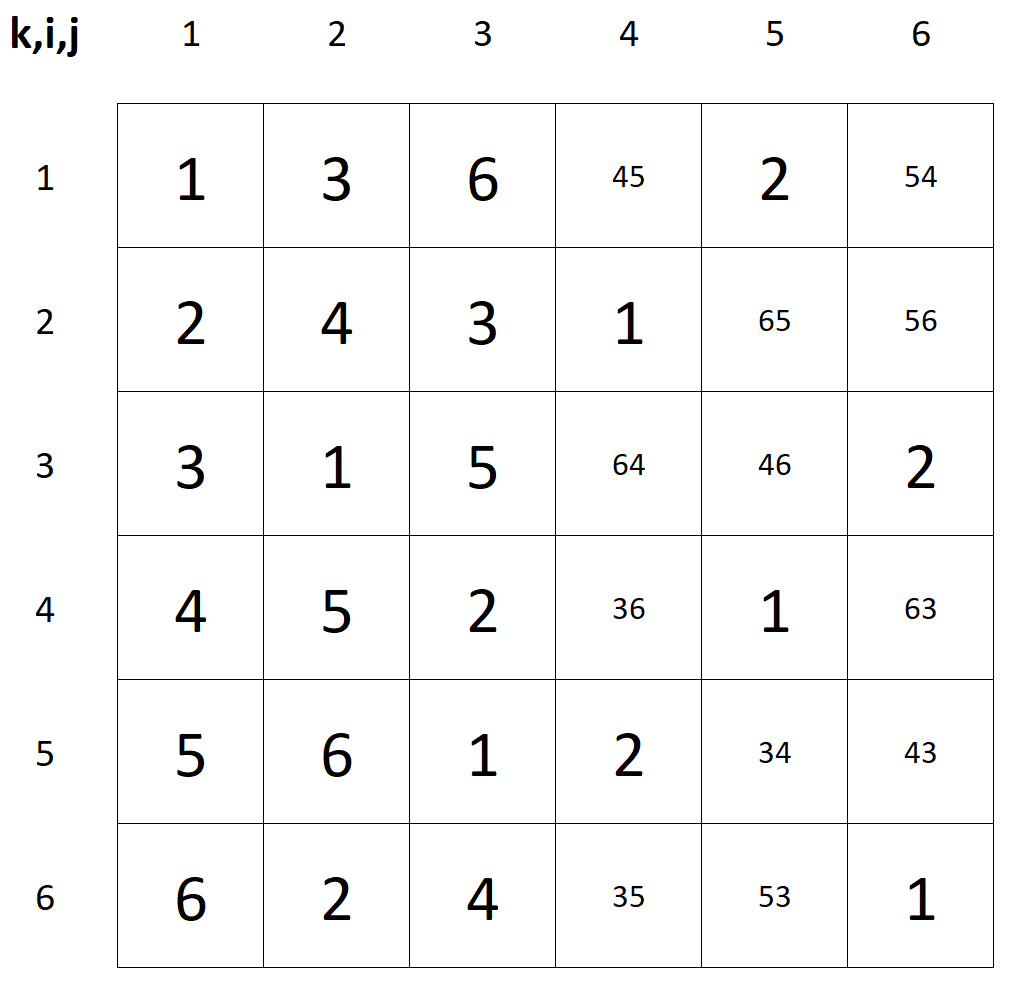}\hfill\includegraphics[width=0.3\textwidth]{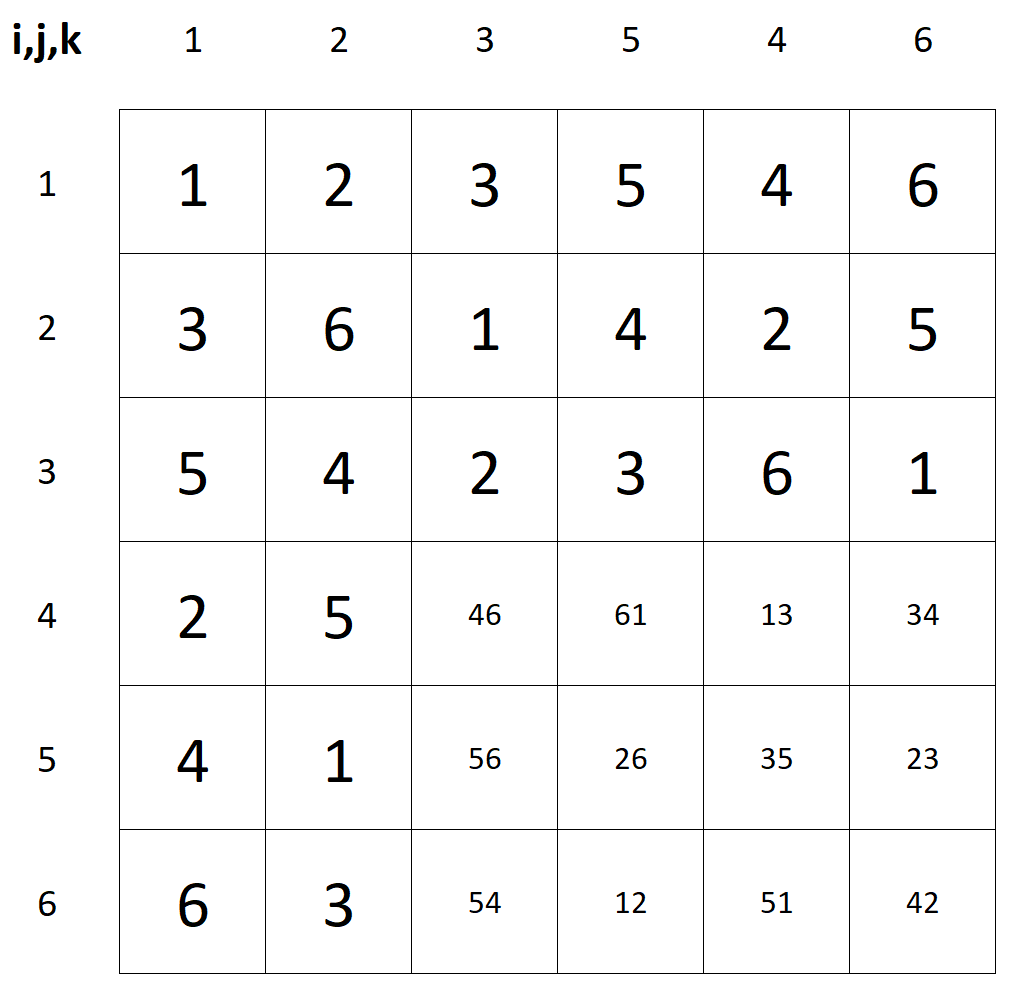}
\hfill\includegraphics[width=0.3\textwidth]{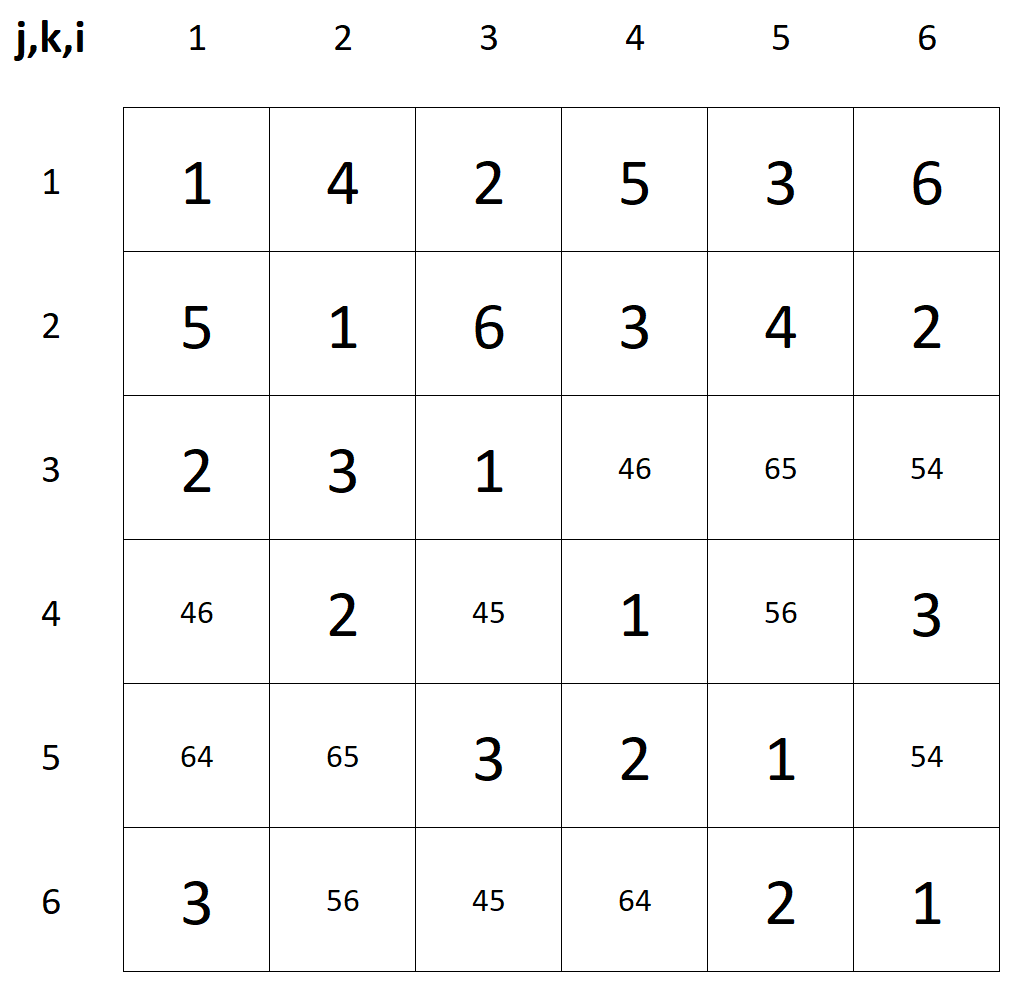}}
\end{tabular}
\caption{}\label{fig3_5}
\end{figure}

\begin{defi}\label{def:dotgraph}
Let $P$ be a PLSC. The \emph{dot graph} of $P$ has the dots of $P$ as its vertices, two
dots being adjacent if and only if some file contains both of them. The
\emph{components of the dot structure} of $P$ are the connected components of the dot
graph.
\end{defi}

\begin{rmrk}\label{rmrk:filecomponent}
All the dots of one file lie in the same component of the dot structure, since any two
of them are adjacent in the dot graph. This is the property we shall use in
Statement~\ref{statm4.19}.
\end{rmrk}

Now let $P$ be a PLSC that has an mCS. Every file that contains a rook contains no dot,
by automatic elimination, and every other file contains exactly two dots. A dot lies in
three files, one in each direction, so it has exactly one neighbour in each direction.
Hence the dot graph of an mCS is $3$-regular, and the three directions (row, column and
symbol files) colour its edges properly with three colours: the edges of one colour form
a perfect matching.
\begin{defi} 
The component of the graph representation of an mCS and the structure of dots corresponding to this component is called \emph{Bivalue Universal Grave}, for short BUG (the name comes from Sudoku).
\end{defi}
\begin{rmrk}
The $3$-dimensional image of the mCS structure of the \emph{GW6} is illustrated on the right-hand side of  Figure~\ref{fig3_6}. The colored version on the left-hand side helps to identify the row layers, layer 6 is green, layer 5 is red and layer 4 is blue, but this coloring has nothing to do with the vertex coloring of the graph. In the graph representation of the mCS structure of the \emph{GW6} the dots of each row are connected, and the row layers are connected to each other (black edges), so the whole graph is a connected. The corresponding graph of the mCS structure in the secondary extension of the Cruse’s square also is connected, therefore, both are BUGs.
\end{rmrk}

\begin{figure}[!htb]
\centering
\includegraphics[width=0.6\textwidth] {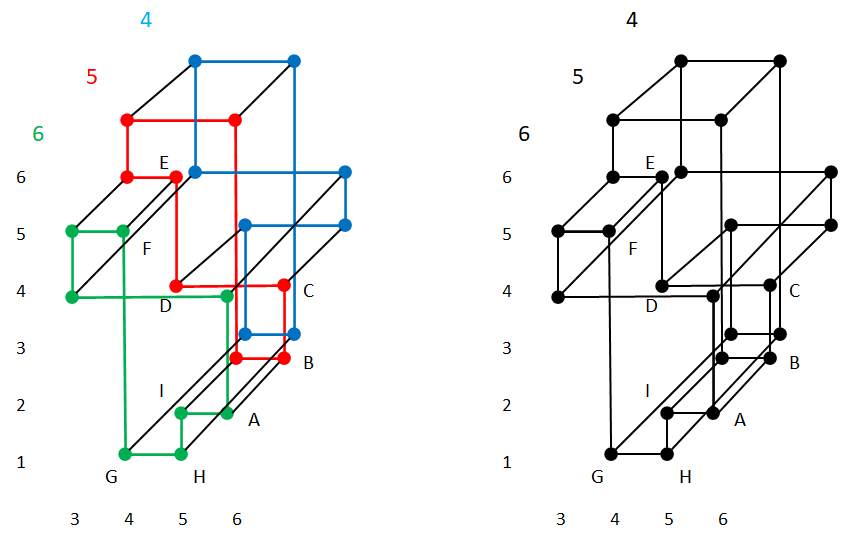}
\caption{}\label{fig3_6}
\end{figure}
\begin{defi}
A BUG has a \emph{solution} or is \emph{solvable} if some of its dots can be replaced with
rooks so that there is exactly one rook in each file of the BUG.
\end{defi}

\begin{rmrk}
A solution automatically uses exactly half of the dots. Let $S$ be the set of dots
replaced by rooks and let $V$ be the set of dots of the BUG. Every file of the BUG is an
edge of the dot graph, so the BUG has $3|V|/2$ files, being $3$-regular. Each dot of $S$
lies in three files and each file contains exactly one dot of $S$, hence
$3|S|=3|V|/2$, that is, $|S|=|V|/2$.
\end{rmrk}
An immediate necessary condition for completion of a PLSC follows.
\begin{theo} [BUG condition]
A necessary condition for the completion of a PLSC is that all BUGs in the secondary extension of the PLSC must be solvable.
\end{theo}
From the graph representation it is evident that a BUG is solvable if and only if the vertices of the corresponding graph can be colored with two colors, which is equivalent to the graph being bipartite, which is true exactly if the graph does not contain a cycle with an odd number of edges (see also Statement~\ref{statm:oddgirth}, which bounds the length of such a cycle from below). There is cycle of length 9 with vertices $A-B-C-D-E-F-G-H-I$ in the Figure~\ref{fig3_6}. Therefore, the \emph{GW6} is not completable.

We have the similar problem in the secondary extension of the Cruse’s square on the left-hand side of Figure~\ref{fig3_7}. If you want to color the vertices of the cycle
\[(1,2,5),(1,6,5),(1,6,4),(5,6,4),(5,2,4),(5,2,6),(1,2,6)
\]
alternately red and blue, you will get stuck at candidate 6 of cell (1,2), because it already has a red and a blue neighbor, so the given cycle is of odd length, namely 7. 
There could exist more cycles is of odd length, another example 
\[
(1,2,5),(6,2,5),(6,2,4),(6,5,4),(1,5,4),(1,5,6),(1,2,6)
\]
is shown on the right-hand side of Figure~\ref{fig3_7}.
Thus, Cruse’s square has no completion either. 
\begin{figure}[!htb]
\centering
\includegraphics[width=0.8\textwidth] {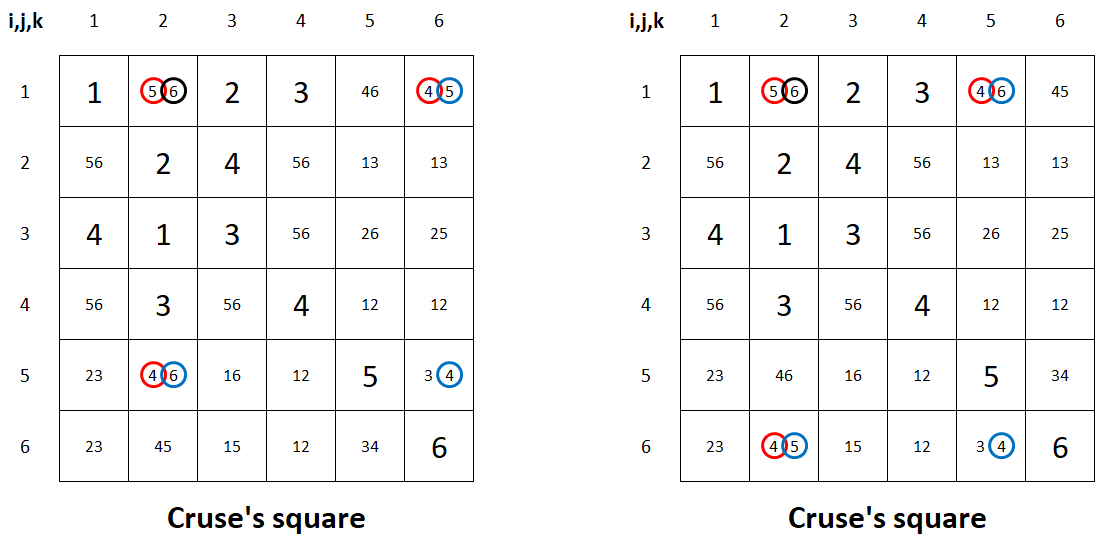}
\caption{}\label{fig3_7}
\end{figure}

\begin{statm}\label{statm:oddgirth}
Every odd cycle of a BUG has length at least $7$, and this bound is attained.
\end{statm}

\begin{proof}
Colour each edge of the dot graph by the direction of the file that produced it. As
observed above this is a proper $3$-edge-colouring, so two consecutive edges of a cycle
have different colours: a dot has only one neighbour in each direction, so a cycle
cannot enter and leave a dot along the same direction.

Consider a cycle and, for $i\in\{1,2,3\}$, let $e_i$ be the number of its edges of
colour~$i$. An edge of colour $i$ changes the $i$-th coordinate of the cell and leaves
the other two unchanged. Going once around the cycle, every coordinate must return to
its initial value, so $e_i\neq 1$ for each $i$; thus $e_i=0$ or $e_i\geq 2$.

If exactly two colours occur, they alternate along the cycle, so the cycle has even
length. Hence an odd cycle uses all three colours, and then $e_1,e_2,e_3\geq 2$, so its
length is at least $6$; being odd, it is at least $7$. The cycle of length $7$ exhibited
in Figure~\ref{fig3_7} shows that the bound is attained.%
\end{proof}

We now summarize our observations on BUGs. If, starting from any cell of a BUG, considering one of the two candidates as a symbol, we can solve all the cells of the BUG without inconsistency, then

\goodbreak

\begin{itemize}
\item [--]
starting from the same cell, considering the other candidate as a symbol, we can solve the BUG and each cell of the BUG has the other symbol as during the first solution,
\item [--]
starting from the given cell, the BUG has exactly two disjoint solutions,
\item [--]
connected bipartite graphs are uniquely partitioned into two color classes, so starting from any cell in the BUG we get one of the previous two solutions.
\end{itemize}
\begin{theo} [BUG theorem] 
A BUG has exactly 0 or 2 solutions. 
\end{theo}
\begin{rmrk} 
The statement is obviously also true for Sudoku squares. Sudoku puzzles have precisely one completion, so they cannot contain BUGs. This is the reason why the Sudoku technique BUG+1 came about.
\end{rmrk}
As we mentioned before Cruse~\cite{[1]} proved that the capacity condition holds for the Cruse’s square. To be precise he proved, that the characteristic matrix $M$ of the Cruse’s square can be extended to a triply stochastic matrix $M^*$ by replacing some $0$'s with $1/2$s. The symbol layers of the triply stochastic matrix are depicted by Cruse in his proof. Like all triply stochastic matrices, $M^*$ satisfies the capacity condition and contains all the 1's of $M$ which are corresponding to the rooks of the original square. So, each RBC of the Cruse’s square satisfies the capacity condition, even if you add some $1/2$s to the number of rooks. Easy to check, that the $1/2$s in the Cruse’s proof are in the cells that
contain dots in the proper BUG depicted in the middle square of Figure~\ref{fig3_4}. Thus, each file of the matrix contains either one 1 or two $1/2$s. The other cells of the files
contain 0. Hence, it is clear from the construction, that the proper matrix is triply stochastic. We don't know how Cruse identified the cells that contain $1/2$s in his proof, maybe by ignoring the deceptive candidates that detroy immediately the completability, but his reasoning can be extended in the following way.
\begin{defi}
A component of the dot structure of a PLSC is called \emph{k-uniformly-distributed}  if each file of the component contains exactly $k$ dots for some $k > 1$, or simply \emph{uniformly-distributed} if $k$ has no role.
\end{defi}
\begin{defi}
A PLSC is \emph{uniformly-distributed} if all components of the dot structure are uniformly-distributed, but the values of $k$ may vary from component to component.
\end{defi}
\begin{statm}\label{statm4.19}
Let $P$ be a PLSC. If you can eliminate some dots such that the remaining PLSC is uniformly-distributed and has no eliminated files, then $P$ satisfies the capacity condition.
\end{statm}
\begin{proof}
Let $P'$ be the uniformly-distributed PLSC obtained after the elimination, and build a
non-negative $n\times n\times n$ matrix $M^*$ as follows: put $1$ in every cell holding a
rook, put $1/k$ in every cell holding a dot of a $k$-uniformly-distributed component,
and put $0$ everywhere else.

We claim that every file of $M^*$ sums to $1$. If a file contains a rook, then by
automatic elimination it contains no dot, so its entries are one $1$ and zeros. If a
file contains no rook, then it is not eliminated, so it contains at least one dot; by
Remark~\ref{rmrk:filecomponent} all its dots belong to the same component, say a
$k$-uniformly-distributed one, and then by definition the file contains exactly $k$ of
them, each carrying $1/k$. In both cases the file sums to $1$, so $M^*$ is a triply
stochastic matrix.

Finally, let $(T_0,T_3)$ be any RBC. Applying the Distribution Theorem to $M^*$
(see~\cite{[3]}, where it is proved for $d$-fold stochastic matrices as well), the
entries of $M^*$ in $T_0\cup T_3$ sum to exactly $\capa(T_0,T_3)$. Every rook of $P$
contributes $1$ to this sum and every other entry is non-negative, so the number of
rooks of $P$ in $T_0\cup T_3$ is at most $\capa(T_0,T_3)$. Eliminating dots does not
change the rooks, so this is the number of rooks of $P$ as well, and $P$ satisfies the
capacity condition.%
\end{proof}
\begin{rmrk}
The PLSC derived from secondary extension of the Cruse’s square and the PLSC derive from \emph{GW6} each contain one 2-uniformly-distributed mCSs.
\end{rmrk}
\begin{cor}\label{cor4.21}
The secondary extension of the Cruse’s square, shown in the middle of Figure~\ref{fig3_4} and the PLS GW6, shown in the Figure~\ref{fig2_2}, both satisfy the capacity condition (and still have no completions).
\end{cor}

\FloatBarrier

\bibliographystyle{plain}

\end{document}